\PassOptionsToPackage{a4paper,margin=25mm}{geometry}

\RequirePackage{iftex}
\ifpTeX
  \PassOptionsToPackage{dvipdfmx}{graphicx}
  \PassOptionsToPackage{dvipdfmx}{hyperref}
\fi

\documentclass[12pt]{article}
\usepackage{da-common}
\usepackage{paper2-h4}
\usetikzlibrary{arrows.meta}

\title{The Fourth Geometry II: From Angle Axioms to Metric Foundations}
\author{%
Masanori Nakazato\\
Mita International School of Science, Tokyo, Japan\\
\texttt{masa727axio.math@gmail.com}\\[6pt]
\small Formerly at the Graduate School of Science, Tohoku University (M.Sc.)
}
\date{January 20, 2026}

\begin{document}

\maketitle


         \begin{abstract}
This paper is a sequel to arXiv:2511.01024 (Base 1), where an axiomatic framework for angles
and the foundations of difference–angle geometry were introduced.

In difference–angle geometry, where the difference of slopes of lines is treated as a primary angular quantity (the difference angle), we reconstruct the focal structure of parabolas from a difference–angle–theoretic viewpoint and develop the associated algebraic and analytic structures.
First, we introduce the difference–angle focal function and define the focus of a parabola constructively as its zero set. This approach yields a formulation of the parabolic power that differs from that presented in Base 1.
Next, by interpreting the power as a classical representation of an inner product, we derive a difference–angle version of the parallelogram theorem via a polarization identity, and thereby define the difference–angle inner product as a pseudo–inner product. The robustness of this structure is substantiated by deriving a difference–angle version of Stewart's theorem based solely on computations involving the difference–angle inner product.

Furthermore, we define the parabolic trigonometric functions $\cosp\theta$ and $\sinp\theta$ (together with related functions) associated with a difference angle $\theta$, and show that they satisfy identities corresponding to the first and second cosine laws in Euclidean geometry.
Finally, we reexamine the Cayley–Klein angle and distance derived from Laguerre's formula, and in particular verify that the existing Cayley–Klein angle satisfies the axiomatic system for angles introduced in Base 1. We then show that, in the parabolic limit of the absolute conic, the difference angle and the difference–angle norm arise naturally as the linear degeneration of the logarithmic cross ratio.
\end{abstract}

\begin{notation}
We refer to objects of Euclidean geometry translated into the framework of
difference--angle geometry as \emph{difference--angle} objects.
When no confusion arises, we omit this qualifier.

Throughout this paper we use the following notation:
\begin{itemize}
  \item $\R$ : the set of real numbers.
  \item $x_P$ : the $x$--coordinate of a point $P$.
  \item $\Slp{\ell}$ : the slope of a line $\ell$.
  \item $\measuredangle_{\mathcal{P}} XYZ$ : the difference angle at the vertex $Y$ of the triangle $\triangle_{\mathcal{P}} XYZ$.
  \item $\measuredangle_{\mathcal{P}} A$, $\measuredangle_{\mathcal{P}} B$, $\measuredangle_{\mathcal{P}} C$ : abbreviations of $\measuredangle_{\mathcal{P}} BAC$, $\measuredangle_{\mathcal{P}} CBA$, and $\measuredangle_{\mathcal{P}} ACB$, respectively.
  \item $\theta_A$, $\theta_B$, $\theta_C$ : variables used to denote the above three interior angles concisely in calculations and proofs.
  \item $\langle u,v\rangle_{\mathcal{P}}$, $|u|_{\mathcal{P}}$ : the inner product and the norm in difference--angle geometry.
\end{itemize}

\end{notation}


\section{Introduction}

\subsection{Motivation}
Since classical times, Euclidean geometry has developed refined axiomatizations of distance and area, while angles have often been treated as secondary quantities subordinate to length or rotation.
Motivated by this imbalance, the author proposed an axiomatic treatment of angles as primary quantities and constructed, as a concrete model, the \emph{difference angle} (the difference of slopes) and the associated \emph{Difference--Angle geometry} (hereafter, \emph{DA geometry}); see Nakazato~\cite{Nakazato2025Base1}.

In the present paper (Base~2), we further develop DA geometry and demonstrate how the structural backbone of Euclidean geometry—namely, power of a point, the parallelogram theorem, and the inner product—can be reconstructed from a viewpoint in which the difference angle is taken as fundamental.

From the first half through the final part of this paper, the main results are as follows:
\begin{itemize}
\item[(i)] We introduce the \emph{difference--angle focal function}, which characterizes the focus in DA geometry, and define the focus constructively as its zero set.
\item[(ii)] Based on this construction, we define the \emph{parabolic power} and derive a parabolic power theorem from a formulation different from that presented in Base~1.
\item[(iii)] Interpreting the power as a classical representation of an inner product, we derive a difference--angle version of the parallelogram theorem via a polarization identity, and define the \emph{difference--angle inner product} as a pseudo--inner product.
\item[(iv)] As an application, we derive a difference--angle version of Stewart's theorem using only computations with the difference--angle inner product.
\end{itemize}

Furthermore, one of the central achievements of this paper is the introduction of the \emph{parabolic trigonometric functions} $\cosp$ and $\sinp$.
Their validity is supported by the fact that they satisfy identities corresponding to the first and second cosine laws, and by their ability to reflect the structural features of classical theorems in Euclidean geometry.
In particular, inspired by aspects not explicitly treated in Weiss--Odehnal~\cite{WeissOdehnal2024}, we complement this perspective by deriving the corresponding theorems within the framework of DA geometry.

\subsection{Position within Cayley--Klein Geometry}
The significance of the final part of this paper lies in clarifying the position of DA geometry within the framework of Cayley--Klein (CK) geometry.
In the CK setting, DA geometry arises as the limit in which the absolute conic degenerates to a parabola, known as the \emph{parabolic degeneration}.
Classically, angles are defined via cross ratios based on Laguerre's idea; however, in the parabolic degeneration, the cross--ratio formulas become indeterminate.

In this paper, we first reformulate Laguerre's formula in a modern framework and define the standard CK angle, showing that it satisfies the angle axioms~A1--A5.
Next, for a one--parameter family of absolute conics
\[
Q_t:\ y = \kappa x^2 + t \qquad (\kappa >0),
\]
we define a parabolic angle using suitable isotropic lines and show that the difference angle appears as the linear degeneration limit as $t \to 0$.
Similarly, since the linear degeneration limit of the CK distance yields the difference--angle norm, DA geometry is clearly identified as a degeneration model of CK geometry that is distinct from previously known geometries.

\subsection{Treatment of Ideal Points}
In DA geometry, the structure at infinity plays an essential role.
The line at infinity in the projective plane splits into two distinct ideal points: $\infty_G$, corresponding to the projective reference direction, and $\infty_S$, corresponding to the singular direction.
This splitting gives rise to phenomena not observed in existing geometric frameworks, such as the degeneration of the triangle inequality into equality and the multiplicity of zero vectors.

In this paper, we illustrate this intuition only to the minimal extent necessary, in order to assist the understanding of subsequent theorems.
In particular, the mechanism by which angular quantities converge to the difference angle in the parabolic limit of the Cayley--Klein angle becomes transparent once this structure at infinity is made explicit.

\begin{figure}[H]
  \centering
  \begin{minipage}{0.30\textwidth}
    \centering
    \includegraphics[width=\linewidth]{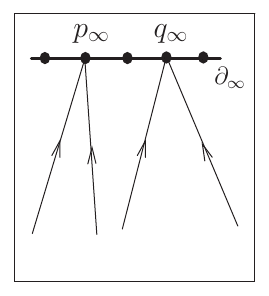}
  \end{minipage}
  \caption{An affine plane with the line at infinity $\ell_{\infty}$.
  Finite segments asymptotically approach two ideal points $p,q \in \ell_{\infty}$.}
 \label{fig:affine-ivt}
\end{figure}

\subsection{Historical Note}
Within the classical framework of CK geometry, when the absolute conic degenerates into a parabola tangent to the line at infinity, the cross--ratio formulas for both distance and angle become indeterminate, and no explicit, coherent model was established to define them simultaneously.
The present work resolves this indeterminacy by linearizing the logarithmic cross ratio, thereby providing an analytic--geometric model in which both distance and angle are consistently defined even when the absolute figure is a parabola.

\begin{figure}[H]
  \centering
  \begin{minipage}{0.30\textwidth}
    \centering
    \includegraphics[width=\linewidth]{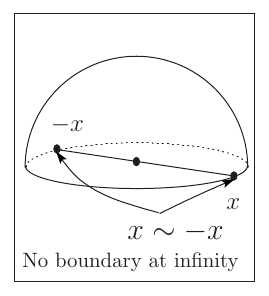}
    \subcaption{}
    \label{fig:elliptic-ivt}
  \end{minipage}
  \begin{minipage}{0.30\textwidth}
    \centering
    \includegraphics[width=\linewidth]{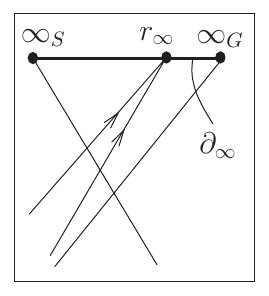}
    \subcaption{}
    \label{fig:diff-ang-ivt}
  \end{minipage}
  \begin{minipage}{0.30\textwidth}
    \centering
    \includegraphics[width=\linewidth]{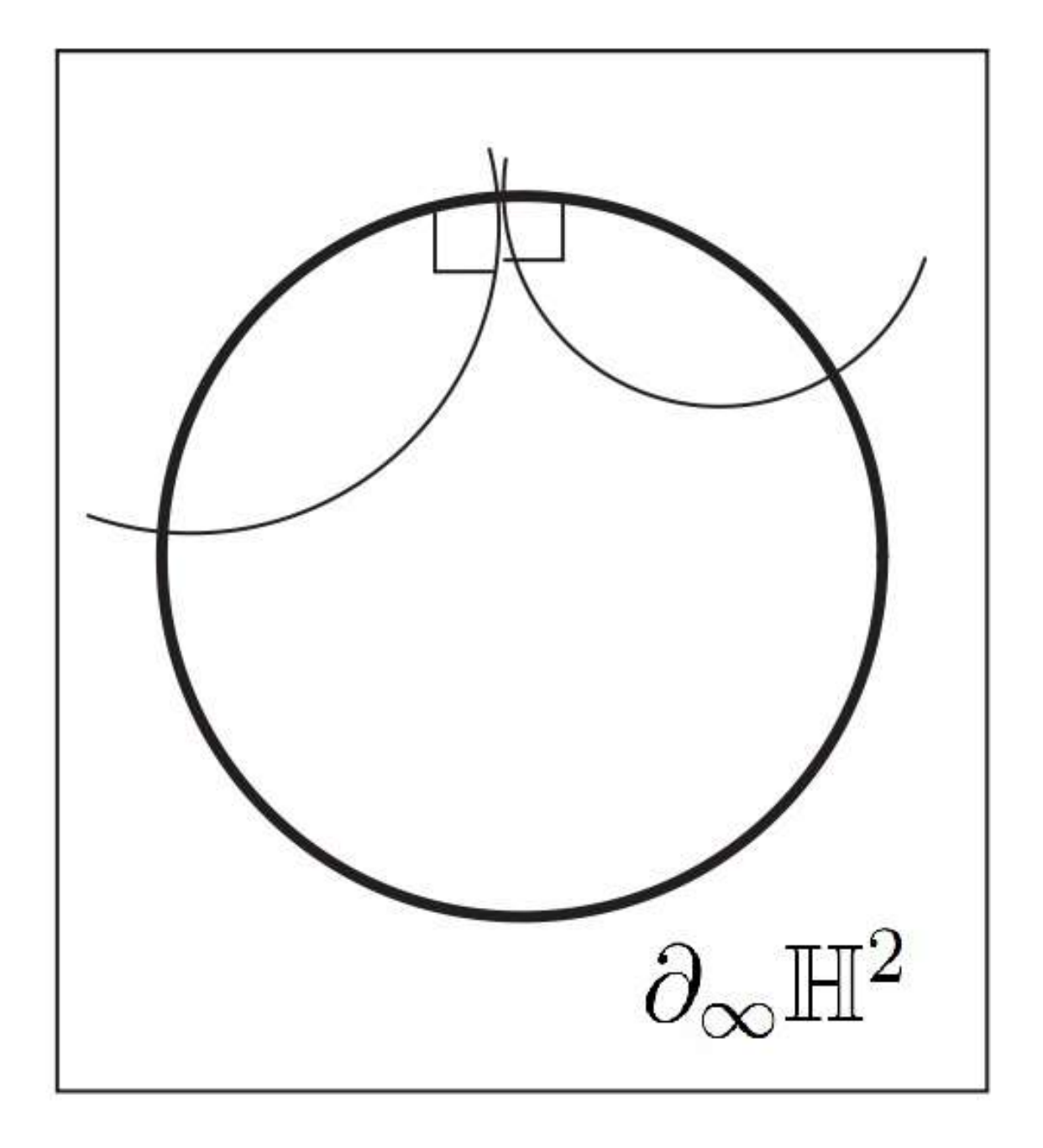}
    \subcaption{}
    \label{fig:hyperbolic-ivt}
  \end{minipage}
  \caption{(a) No real ideal points in elliptic geometry.
  (b) In DA geometry, $\ell_{\infty}$ splits into $\infty_G$ and $\infty_S$.
  (c) In hyperbolic geometry, they coincide, forming a compact ideal boundary.}
\end{figure}


\section{Difference--Angle Focus and Parabolic Loci}

\begin{quote}
In DA geometry, the focus of a parabola is not characterized by a Euclidean distance condition,
but rather by relations among difference angles.
In this section, we derive a fundamental equation describing the difference--angle correspondence
between the focus and the directrix, and show that this equation uniquely determines a parabolic locus.
This notion of focus plays a crucial role in characterizing the parabolic power theorem,
which served as the original motivation for DA geometry.
\end{quote}

\begin{figure}[htbp]
  \centering
  \begin{subfigure}{0.48\textwidth}
    \centering
    \includegraphics[scale=0.8]{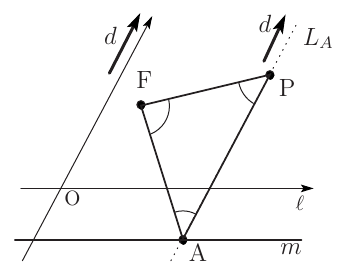}
    \subcaption{}
   \label{fig:parabola-focus}
  \end{subfigure}
  \hfill
  \begin{subfigure}{0.48\textwidth}
    \centering
    \includegraphics[scale=0.8]{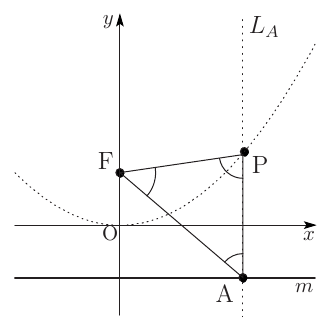}
    \subcaption{} 
   \label{fig:parabola-focus-normal}
  \end{subfigure}
  \caption{(a) General configuration of the difference--angle focus theorem.
  (b) Normalized configuration used in the proof, with $A=(a,-q)$, $F=(0,q)$, and the parabola
  $y=x^{2}/(4q)$ (dotted).}
\end{figure}

For a point $P=(x,y)$ on a parabola, it is required to rewrite the defining relation in terms of
the difference angle $\Pangle$.
On the other hand, since the focus lies on a singular line passing through the vertex of the parabola,
expressing the parabola $y=x^{2}$ in terms of the difference angle $\Pangle$ requires rewriting its defining equation
using only slopes of lines, that is, linear expressions.
However, $y=x^{2}$ is a quadratic equation, and its degree does not match that of the difference angle,
which is defined as a difference of slopes.

To resolve this mismatch, we first abstract the $y$--coordinate of the focus and rewrite the equation
of the parabola in the form
\[
  4qy = x^{2}.
\]
In this form, both sides are expressed as products of linear terms, making it possible to represent
the equation purely in terms of difference angles.

\begin{definition}[Difference--Angle Focal Equation]\label{def:diff-focal-eq}
Let $(\ell,d)$ be a projective reference structure on the plane, and let $F$ be a finite point.
For an arbitrary point $A$, denote by $L_A$ the singular line through $A$ (parallel to the direction $d$).
For $P\in L_A$, define
\begin{equation}\label{eq:diff-focal-function}
\mathcal{F}_A(P)
=1-\Pangle FAP\,
   \bigl(\Pangle APF-\Pangle AFP\bigr),
\end{equation}
and call $\mathcal{F}_A(P)$ the \emph{difference--angle focal function}.
In particular, when
\begin{equation}\label{eq:diff-focal-eqn}
\mathcal{F}_A(P)=0,
\end{equation}
we call this equation the \emph{difference--angle focal equation}.
\end{definition}

\begin{proposition}[Uniqueness of the Solution of the Focal Equation]\label{prop:uni-exist-diff-parabola-focus}
Under \Cref{def:diff-focal-eq}, for any given point $A$,
if $F\notin L_A$, then there exists a unique point $P\in L_A$ satisfying
\eqref{eq:diff-focal-eqn}.
\end{proposition}

\begin{proof}
We normalize the configuration by setting $F=(0,q)$, $A=(a,-q)$, and $P=(a,y)$
(the degenerate case $a=0$ has been treated separately).
We proceed by cases.

\medskip
\noindent\textit{Case $a>0$.}
In this case, the segment $AP$ is parallel to the direction $d$, hence $\Slp{AP}=0$.
Moreover,
\[
  \Slp{FA}=-\frac{2q}{a}\ (<0),\qquad
  \Slp{FP}=\frac{y-q}{a}
  \ \begin{cases}
    >0 & (y>-q),\\[2pt]
    <0 & (y<-q).
  \end{cases}
\]
By the definition of the difference angle, we obtain
\[
  \Pangle APF=-\frac{y-q}{a},\qquad
  \Pangle AFP=\frac{y+q}{a},\qquad
  \Pangle FAP=-\frac{2q}{a}.
\]
Substituting these expressions into $\mathcal{F}_A(P)=0$, that is,
\(
1-\Pangle FAP\,
      (\Pangle APF-\Pangle AFP)=0,
\)
we obtain
\[
  1+\frac{2q}{a}\Bigl(-\frac{y-q}{a}-\frac{y+q}{a}\Bigr)=0
  \iff 1-\frac{4qy}{a^2}=0
  \iff 4qy=a^{2}.
\]
Writing $x=a$, this yields $4qy=x^{2}$.
Although the sign of $\Slp{FP}$ changes according to whether $y>-q$ or $y<-q$,
the resulting equation is identical in both cases.

\medskip
\noindent\textit{Case $a<0$.}
A completely analogous computation yields the same equation $4qy=a^{2}$,
since the sign differences cancel out after clearing denominators.
\end{proof}

\begin{maintheorem}[Focal Equation and Parabola]\label{mthm:diff-parabola-focus}
Let $(\ell,d)$ be a projective reference structure on the plane,
and let $F$ be a finite point.
Fix a line $m$ parallel to $\ell$, and for each point $A\in m$,
let $P\in L_A$ satisfy \eqref{eq:diff-focal-eqn}.
Then, as $A$ varies along $m$, the locus of $P$ is a parabola whose axis is parallel to the direction $d$.
\end{maintheorem}

\begin{proof}
By \Cref{prop:uni-exist-diff-parabola-focus}, for each point $A$
there exists a unique point $P=(x,y)\in L_A$ satisfying
\[
  4qy = x^{2}.
\]
As $A$ varies along $m$, the value $x$ ranges over all real numbers.
Hence the locus of $P$ coincides with the curve defined by
\[
  4qy = x^{2}.
\]
This curve is precisely the parabola with focus $(0,q)$ and directrix $y=-q$,
and its axis is parallel to the direction $d$.
\end{proof}

\begin{remark}
Conversely, if $P$ satisfies $y=\dfrac{x^{2}}{4q}$, then reversing the above chain of equivalences
immediately yields \eqref{eq:diff-focal-eqn}.
Therefore, \eqref{eq:diff-focal-eqn} is equivalent to the parabolic condition.
\end{remark}

\begin{remark}[Normalization by Vertical Translation]
In the above proof, we have assumed $A=(a,-q)$ and $F=(0,q)$ for convenience.
However, this normalization can always be achieved by a translation in the direction $d$,
namely $(x,y)\mapsto(x,y+t)$.
Since such translations preserve slopes, they do not affect the invariance of
\eqref{eq:diff-focal-eqn}.
Therefore, in the general case $A=(a,q')$ and $F=(0,q)$,
the conclusion remains unchanged, with only a vertical shift of the vertex of the parabola.
\end{remark}


\section{Parabolic Power and Its Associated Structures}

\begin{quote}
In this section, we investigate the parabolic power---the driving force behind DA geometry---and the surrounding structures.
The parabolic power theorem\footnote{The fact itself was communicated to the author, but the original discoverer is unknown.}
was already treated in Nakazato~\cite{Nakazato2025Base1}; nevertheless, we restate it here.
The reason is that, although the theorem holds in a form analogous to that in Euclidean geometry,
the mechanism determining its value is essentially different.
Indeed, in Euclidean geometry the power of a point is determined by the radius of the circle and the distance from the point to the center,
whereas in the parabolic setting no such relation is available.

For a circle, the center $O$ plays a privileged role in determining the power,
while for a parabola the corresponding privileged point is the focus $F$.
Accordingly, by replacing the Euclidean ``center'' viewpoint with the focus, we reexamine the value of the \emph{parabolic power} in DA geometry.

We first derive the parabolic power theorem as a basic result based on AA-similarity,
and then study the mechanism determining its value, the interior/exterior classification,
and related auxiliary structures.
\end{quote}

\begin{figure}[htbp]
  \centering
  \begin{minipage}{0.32\textwidth}
    \centering
    \includegraphics[width=\linewidth]{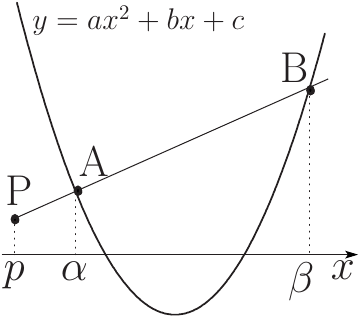}
    \subcaption{}
    \label{fig:Parabolic-power-ext}
  \end{minipage}
  \begin{minipage}{0.32\textwidth}
    \centering
    \includegraphics[width=\linewidth]{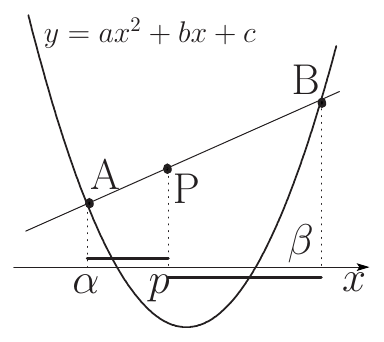}
    \subcaption{}
    \label{fig:Parabolic-power-int}
  \end{minipage}
  \begin{minipage}{0.32\textwidth}
    \centering
    \includegraphics[width=\linewidth]{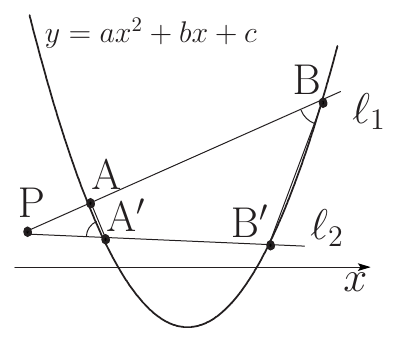}
    \subcaption{}
    \label{fig:Parabolic-power-proof}
  \end{minipage}
 \caption{
(a) A point $P$ on the side not containing the focus.
(b) A point $P$ on the side containing the focus.
(c) Configuration for the proof, showing the two secants $\ell_1,\ell_2$ and points $A,B,A',B'$.
}
\end{figure}

\begin{theorem}[Parabolic Power Theorem]\label{thm:parabolic-power}
Let $\CMC$ be a parabola and let $P$ be a point.
For a line $\ell$ through $P$ that meets $\CMC$ at two points $A,B$, the value of
$|PA|_{\mathcal{P}}\,|PB|_{\mathcal{P}}$ is independent of the choice of $\ell$.
\end{theorem}

\begin{proof}
We show that the values determined by two distinct secants coincide.

Let $A,B$ be the intersection points of $\ell_1$ with $\CMC$, and let $A',B'$ be those of $\ell_2$ with $\CMC$.
The angle $\Pangle APA'$ is common. Independently of the positions of $A',B'$,
the constancy of the parabolic inscribed angle (equivalently, a property of parabolic cyclic quadrilaterals) yields
\[
\Pangle PAA'=\Pangle PB'B.
\]
Hence
\[
\Ptri PAA' \sim_{\mathcal{P}} \Ptri PB'B.
\]
Therefore, the product of the corresponding side ratios agrees, and we obtain
\[
|PA|_{\mathcal{P}} : |PA'|_{\mathcal{P}}=|PB'|_{\mathcal{P}} : |PB|_{\mathcal{P}}
\iff |PA|_{\mathcal{P}} |PB|_{\mathcal{P}}=|PA'|_{\mathcal{P}} |PB'|_{\mathcal{P}}.
\]
This proves the claim.
\end{proof}

Although the parabolic power theorem has an extremely elementary nature, the value of the power
cannot be read off transparently from purely geometric properties.
For comparison, for a circle $C$ with center $O$ and radius $r$,
the Euclidean power $\Pi(P;C)$ of a point $P$ satisfies, writing $OP=d$,
\[
  \Pi(P;C)=d^{2}-r^{2}.
\]
To emphasize its correspondence with a slope-based structure, we rewrite it in a dimensionless form as
\[
  \Pi(P;C)
  =d^{2}\left(1-\left(\frac{r}{d}\right)^{2}\right)
  \quad (d\neq 0).
\]

Now fix a projective reference structure $(\ell,d)$.
Applying the difference--angle focal equation
\[
  \mathcal F_A(P)
  =1-\Pangle FAP\,
     \bigl(\Pangle APF-\Pangle AFP\bigr)=0,
\]
we denote by $\CMC$ the parabola determined by the focus $F$ and the directrix $m$.

In what follows, we define the power with respect to this parabola,
prove its constancy, and show that the sign of the power determines the position of $P$.
For this purpose, we first specify the interior and exterior of a parabola.

\begin{figure}[htbp]
  \centering
  \begin{minipage}{0.32\textwidth}
    \centering
    \includegraphics[width=\linewidth]{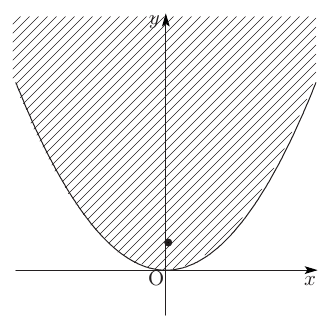}
    \subcaption{}
    \label{fig:focal-side}
  \end{minipage}
  \begin{minipage}{0.32\textwidth}
    \centering
    \includegraphics[width=\linewidth]{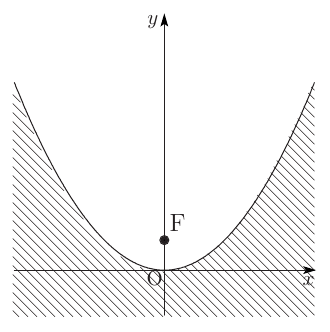}
    \subcaption{}
    \label{fig:non-focal-side}
  \end{minipage}
 \caption{
(a) The focal side of the parabola (shaded).
(b) The non-focal side of the parabola (shaded).
}
\end{figure}

\begin{definition}[Interior and Exterior of a Parabola]\label{def:int-ext-parabola}
Let $d$ be the axis direction of a parabola $\CMC$.
We call the side containing the focus the \emph{interior} of $\CMC$, and the opposite side the \emph{exterior}.
We write
\[
  P\in\Int(\CMC)\quad\text{or}\quad P\in\Ext(\CMC).
\]
\end{definition}

\begin{remark}
This definition distinguishes the interior and exterior of a parabola by the sign of the parabolic power.
Its structure agrees with Klein's notion of the interior and exterior of an absolute conic (fundamental conic);
see Klein~\cite{Klein1871}.
\end{remark}

In what follows, for points $P$ and $X$ we use the $x$--coordinate difference
$s_P(X):=x_X-x_P$.

\begin{maintheorem}[Main Theorem on Parabolic Power]\label{mthm:parabolic-power-main}
For a point $P$ and a parabola $\mathcal C$, define the \emph{parabolic power} by
\[
   \Pi_{\mathcal P}(P;\mathcal C)
  := \dnorm{PF}^{2}\,\mathcal F_{Q}(P).
\]
Then, for any secant $\ell$ through $P$ with intersection points $A,B$ with $\mathcal C$, we have
\[
  \Pi_{\mathcal P}(P;\mathcal C)
   = s_P(A)s_P(B).
\]
Moreover, the sign of this value classifies the position of $P$ as
\[
    \begin{cases}
      >0 &\Longleftrightarrow P\in\Ext(\mathcal C),\\
      =0 &\Longleftrightarrow P\in\mathcal C,\\
      <0 &\Longleftrightarrow P\in\Int(\mathcal C).
    \end{cases}
\]
In particular, $\Pi_{\mathcal P}(P;\mathcal C)$ is independent of the choice of the secant $\ell$,
and we call it the \emph{parabolic power of $\mathcal C$ at $P$}.
When no confusion arises, we abbreviate $\Pi_{\mathcal P}(P):=\Pi_{\mathcal P}(P;\mathcal C)$.
\end{maintheorem}

\begin{proof}
We normalize to $\mathcal C:\ y=x^{2}$, and write $P=(u,v)$, with intersection points
$A=(a,a^{2})$ and $B=(b,b^{2})$.

As in the normalization used in the proof of \cref{prop:uni-exist-diff-parabola-focus},
we take the focus $F=(0,\frac14)$ and the point $Q=(u,-\frac14)$ on the directrix aligned with $P$.
Under this setup, a direct computation from the definition of the difference angle yields
\[
  \mathcal{F}_{Q}(P)=1-\frac{v}{u^{2}}.
\]
By the definition of the difference--angle norm in the normalized setting,
we have $\dnorm{PF}^{2}=(x_P-x_F)^2=u^2$, and hence
\[
 \Pi_{\mathcal P}(P;\mathcal C)
  = \dnorm{PF}^{2}\,\mathcal F_{Q}(P)
  =u^2\left(1-\frac{v}{u^2}\right)=u^2-v.
\]
Since $s_P(A)=x_A-x_P=a-u$ and $s_P(B)=x_B-x_P=b-u$, it suffices to show that
\[
  s_P(A)s_P(B) = (a-u)(b-u).
\]

Using the parabola identity
\[
  x^{2}=(x-a)(x-b)+(a+b)x-ab,
\]
and substituting $x=u$, we obtain
\[
  u^{2}=(u-a)(u-b)+(a+b)u-ab.
\]
On the other hand, the line $AB$ has equation $y=(a+b)x-ab$.
Since $P=(u,v)$ lies on $AB$, we have
\[
  v=(a+b)u-ab.
\]
Comparing these two equations yields
\[
  (a-u)(b-u)=u^{2}-v.
\]
Therefore
\[
  s_P(A)s_P(B)=u^{2}-v,
\]
and hence
\[
  \Pi_{\mathcal P}(P;\mathcal C)
  =s_P(A)s_P(B).
\]

\medskip
Finally, the sign of $s_P(A)s_P(B)=u^{2}-v$ agrees with whether $P$ lies in the
exterior ($u^{2}-v>0$), on the parabola ($u^{2}-v=0$), or in the interior ($u^{2}-v<0$)
of $y=x^{2}$.
This completes the proof.
\end{proof}

\begin{remark}
For the general parabola $y=\kappa x^2$, we obtain
\[
  \Pi_{\mathcal P}(P)=u^2-\frac{v}{\kappa},
\]
but in this section we work with the normalization $\kappa=1$.
\end{remark}

\begin{remark}[General Requirements for a ``Power Theorem'']
The conditions one may reasonably require in order to call a statement a ``power theorem''
can be summarized as follows:
\begin{enumerate}
\item For any curve $C$ and a fixed point $P$, there exists a line $L$ through $P$ meeting $C$
      in two points $A,B$ (including the tangent case).
\item For $P,A,B$, a quantity is defined by a common rule, and the product of the two values
      (called the ``power'') is independent of the choice of $L$.
\item The sign of the product classifies the position of $P$ as lying inside $C$, outside $C$,
      or on $C$.
\end{enumerate}
The Euclidean power of a point and the parabolic power in the present framework are among the simplest examples satisfying these three requirements.
Within a similar framework, one can also define a \emph{hyperbolic power}.
A systematic study of the geometric structure based on the latter is expected to be a key step
toward clarifying the connection between DA geometry and hyperbolic geometry.
\end{remark}

One of the classical consequences of a power theorem is the existence of the radical axis and the radical center.
Given two figures, the locus of points whose powers with respect to the two figures are equal is called the radical axis.
For three figures in a suitable mutual position, if the three pairwise radical axes are concurrent,
their intersection point is called the radical center; a typical situation is when each pair of figures intersects.

\begin{figure}[htbp]
  \centering
  \begin{minipage}{0.32\textwidth}
    \centering
    \includegraphics[width=\linewidth]{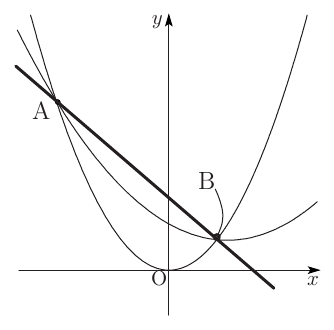}
    \subcaption{}
    \label{fig:radical-axis}
  \end{minipage}
  \begin{minipage}{0.32\textwidth}
    \centering
    \includegraphics[width=\linewidth]{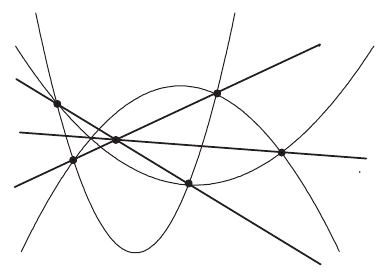}
    \subcaption{}
    \label{fig:radical-center-ppm}
  \end{minipage}
  \begin{minipage}{0.32\textwidth}
    \centering
    \includegraphics[width=\linewidth]{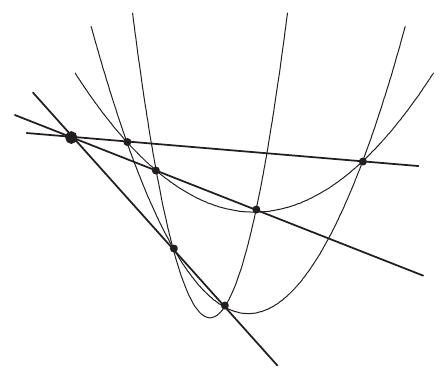}
    \subcaption{}
    \label{fig:radical-center-ppp}
  \end{minipage}
 \caption{
(a) The parabolic radical axis of two parabolas with parallel axes.
(b) The parabolic radical center arising when two quadratic coefficients
    share the same sign and the third has the opposite sign.
(c) The parabolic radical center in the case where all three quadratic
    coefficients have the same sign.
}
\end{figure}

\begin{theorem}[Parabolic Radical Axis]\label{thm:parabolic-radical-axis}
Under a projective reference structure $(\ell,d)$, let $\CMC$ and $\DMC$ be two parabolas whose axes are parallel to $d$,
and assume that they intersect at two distinct points $A,B$.
Then, for any point $P$ on the line $\ell_{AB}$ containing the common chord $AB$, we have
\[
  \Pi_{\mathcal P}(P;\CMC)=\Pi_{\mathcal P}(P;\DMC).
\]
Hence $\ell_{AB}$ is the parabolic radical axis of $\CMC$ and $\DMC$.
\end{theorem}

\begin{proof}
For $P=A$ or $P=B$, both powers are $0$ and the claim is trivial.
Assume $P\neq A,B$.
Under the normalization $\mathcal C:\ y=x^2$, we have
\[
  |\Pi_{\mathcal P}(P;\CMC)|
  =|PA|_{\mathcal P}\,|PB|_{\mathcal P}
  =|\Pi_{\mathcal P}(P;\DMC)|
\]
since the chord endpoints $A,B$ coincide for both parabolas.
Moreover, if $P\in[AB]$ then $P\in\Int$ for both parabolas, while if $P\notin[AB]$ then $P\in\Ext$.
By the sign correspondence in \Cref{mthm:parabolic-power-main}, the signs therefore agree
(because the focus-side half-planes of the two parabolas coincide).
Hence the equality holds.
\end{proof}

\begin{remark}[Degenerate Cases]
If the two parabolas are tangent (i.e.\ they have a single intersection point), then the radical axis is the common tangent line at the tangency point.
If the two curves coincide, then the powers agree at every point, and the radical axis (as for circles) is excluded by definition.
\end{remark}

\begin{theorem}[Parabolic Radical Center]\label{thm:parabolic-radical-center}
Let $\CMC,\DMC,\EMC$ be three parabolas whose axes are parallel to $d$.
Assume that each pair has two distinct intersection points (so that each pair has a common chord),
and let $\ell_{\CMC\DMC},\ell_{\DMC\EMC},\ell_{\EMC\CMC}$ denote the three parabolic radical axes.
Then these three lines are concurrent.
Their intersection point is called the \emph{parabolic radical center}.
\end{theorem}

\begin{proof}
Let $P\in\ell_{\CMC\DMC}\cap\ell_{\DMC\EMC}$.
Then
\[
 \Pi_{\mathcal P}(P;\CMC)=\Pi_{\mathcal P}(P;\DMC),\qquad
 \Pi_{\mathcal P}(P;\DMC)=\Pi_{\mathcal P}(P;\EMC).
\]
Hence $\Pi_{\mathcal P}(P;\CMC)=\Pi_{\mathcal P}(P;\EMC)$, and therefore $P\in\ell_{\EMC\CMC}$.
Thus the three lines are concurrent.
\end{proof}

While the difference--angle focal equation characterizes parabolas as a condition defining the curve,
the main theorem on parabolic power provides its quantitative aspect.
The correspondence between these two viewpoints forms the basis of the parabolic structure in DA geometry.


\section{Difference--Angle Parallelogram Law}

\begin{quote}
In this section, we present the \emph{Difference--Angle Parallelogram Law}, which symbolizes
the norm structure in DA geometry.

For the basic concepts and notation, we refer to~\cite{Nakazato2025Base1}.
A difference--angle triangle $ABC$ in DA geometry is denoted by $\Ptri ABC$,
and the distance between two points $X,Y$ defined from difference angles (the difference--angle norm)
is written as $|XY|_{\mathcal P}$.
\end{quote}

\begin{figure}[htbp]
  \centering
  \begin{subfigure}{0.48\textwidth}
    \centering
    \includegraphics[scale=1.0]{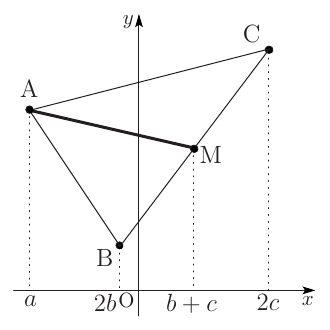}
    \subcaption{}
    \label{fig:parallelo-thm1}
  \end{subfigure}
  \hfill
  \begin{subfigure}{0.48\textwidth}
    \centering
    \includegraphics[scale=1.0]{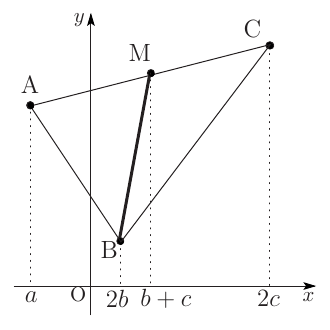}
    \subcaption{}
   \label{fig:parallelo-thm2}
  \end{subfigure}
  \caption{Difference--Angle Parallelogram Law.
  (a) The Ceva line from a positive difference--angle vertex is a median.
  (b) The Ceva line from a negative difference--angle vertex is a median.}
\end{figure}

\begin{maintheorem}[Difference--Angle Parallelogram Law]\label{mthm:DA-Parallelogram-Law}
In a difference--angle triangle $\Ptri ABC$, let $M$ be the midpoint of $BC$.
Then
\[
  {|AB|_{\mathcal P}}^2 + {|AC|_{\mathcal P}}^2
    = 2\Bigl({|AM|_{\mathcal P}}^2 + {|MB|_{\mathcal P}}^2\Bigr).
\]
\end{maintheorem}

\begin{proof}
Since we will introduce the difference--angle inner product later, we verify this theorem directly by a coordinate computation.
Let the $x$--coordinates of the vertices of $\Ptri ABC$ be
$a,\,2b,\,2c$ with $a<2b<2c$.
Then
\[
{|AB|_{\mathcal P}}^2 + {|AC|_{\mathcal P}}^2
 = (2b-a)^2 + (2c-a)^2
 = 2a^2 - 4a(b+c) + 4(b^2 + c^2).
\]
Since the $x$--coordinates of $B$ and $C$ are $2b$ and $2c$, respectively, the midpoint $M$ has $x$--coordinate $b+c$.
Hence $|AM|_{\mathcal P}=|b+c-a|$ and $|MB|_{\mathcal P}=|(b+c)-2b|=|c-b|$.
Therefore
\begin{align*}
{|AM|_{\mathcal P}}^2 + {|MB|_{\mathcal P}}^2
 &= (b+c-a)^2 + (c-b)^2 \\
 &= a^2 - 2a(b+c) + 2(b^2 + c^2).
\end{align*}
Comparing both sides yields
\[
{|AB|_{\mathcal P}}^2 + {|AC|_{\mathcal P}}^2
   = 2\bigl({|AM|_{\mathcal P}}^2 + {|MB|_{\mathcal P}}^2\bigr),
\]
as desired.
\end{proof}

\begin{remark}[On Singular Medians]
In the above proof we used the normalization $x_A=a$, $x_B=2b$, $x_C=2c$ with $a<2b<2c$.
Then the midpoint $M$ has $x$--coordinate $b+c$, and hence $a<b+c$ always holds.
Therefore the median $AM$ can never be parallel to the projective direction $d$
(i.e.\ it cannot be a singular line).

At a negative difference--angle vertex of a difference--angle triangle,
the median becomes a singular line in the difference--angle isosceles case.
Even in that case, however, the parallelogram identity itself follows from the same computation,
and the conclusion remains unchanged.
\end{remark}

This identity forms a basis for the notion of distance in DA geometry and enables the definition of an inner product.
In other words, it confirms that a structure analogous to the Euclidean parallelogram law
also holds for the difference--angle norm.

Thus, the Difference--Angle Parallelogram Law serves as a fundamental identity guaranteeing
the existence of an inner--product structure in DA geometry.


\section{Difference--Angle Inner Product and Its Structure}

\begin{quote}
Since the norm $|\cdot|_{\mathcal P}$ in DA geometry is degenerate,
a usual inner product in the Euclidean sense does not exist.
However, because the parallelogram identity
(\Cref{mthm:DA-Parallelogram-Law}) holds,
a bilinear form can be uniquely reconstructed from the norm
via the Jordan--von Neumann polarization formula.
In this section, we define the difference--angle inner product
induced by this polarization and organize its basic structure.
\end{quote}

\subsection{Degenerate Pseudo--Inner Product}\label{subsec:degenerate-inner}
In general, an inner product on the vector space $V$ of DA geometry
is characterized by the following axioms, allowing degeneracy.

\begin{definition}[Degenerate Pseudo--Inner Product]\label{def:degenerate-inner}
A mapping $\langle\cdot,\cdot\rangle_{\mathcal P}:V\times V\to\mathbb R$
is called a \emph{degenerate pseudo--inner product} if it satisfies:
\begin{enumerate}[label=(P\arabic*)]
  \item \textbf{(Bilinearity)}
        $\langle au+bv,w\rangle_{\mathcal P}
         =a\langle u,w\rangle_{\mathcal P}
          +b\langle v,w\rangle_{\mathcal P}$.
  \item \textbf{(Symmetry)}
        $\langle u,v\rangle_{\mathcal P}
         =\langle v,u\rangle_{\mathcal P}$.
  \item \textbf{(Positive Semidefiniteness)}
        $\langle u,u\rangle_{\mathcal P}\ge 0$.
  \item \textbf{(Null Space)}
        The set of all vectors $u\neq 0$ satisfying
        $\langle u,u\rangle_{\mathcal P}=0$,
        denoted by $Z_{\mathcal P}$,
        forms a linear subspace.
\end{enumerate}
\end{definition}

\begin{remark}[Quotient by the Null Space and Positive Definiteness]
\label{rem:quotient-space-inner}
On the quotient space $\bar V=V/Z_{\mathcal P}$ obtained by collapsing the null space,
the induced form
\[
  \overline{\langle u,v\rangle}_{\mathcal P}
  :=\langle u,v\rangle_{\mathcal P}
\]
is a well--defined positive definite inner product.
In the standard model, $\bar V$ can be identified with the $x$--axis, and
\[
  \overline{\langle u,v\rangle}_{\mathcal P}
  = \bar x_u\,\bar x_v
\]
holds.
\end{remark}

By \Cref{mthm:DA-Parallelogram-Law}, the norm satisfies
\[
  |u+v|_{\mathcal P}^2 + |u-v|_{\mathcal P}^2
  = 2|u|_{\mathcal P}^2 + 2|v|_{\mathcal P}^2.
\]
Hence, by the Jordan--von Neumann theorem,
a bilinear form can be uniquely recovered from the norm.

\begin{definition}[Difference--Angle Inner Product]\label{def:diff-inner}
For vectors $u,v\in V$, define
\[
  \langle u,v\rangle_{\mathcal P}
  := \frac12\Bigl(|u+v|_{\mathcal P}^2
                 - |u|_{\mathcal P}^2
                 - |v|_{\mathcal P}^2\Bigr).
\]
This bilinear form is called the \emph{difference--angle inner product}.
\end{definition}

\begin{proposition}[Pseudo--Inner Product Property via Polarization]
\label{prop:parabolic-polarization}
The mapping $\langle\cdot,\cdot\rangle_{\mathcal P}$ defined in
\Cref{def:diff-inner} satisfies all axioms {\rm(P1)}--{\rm(P4)}
of a degenerate pseudo--inner product.
\end{proposition}

\begin{proof}
For a norm satisfying the parallelogram identity,
it is a standard result of the Jordan--von Neumann theorem
that a bilinear form can be uniquely reconstructed by the polarization formula.
Moreover,
\(
  \langle u,u\rangle_{\mathcal P}
   = |u|_{\mathcal P}^2 \ge 0
\),
so positive semidefiniteness holds.
That the set of vectors satisfying $|u|_{\mathcal P}=0$
forms a linear subspace follows immediately from the definition of the norm.
\end{proof}

\begin{remark}[Terminology]
Although the difference--angle inner product is degenerate,
we shall simply call it an ``inner product'' when no confusion arises.
\end{remark}

\subsection{Basic properties: equality in Cauchy--Schwarz and degenerate cosine}
\label{subsec:basic-properties}

The most important feature of the inner product in \emph{DA geometry} is that the
Cauchy--Schwarz ``inequality'' always holds with equality.
This is consistent with the fact that the triangle inequality for the
difference--angle norm is always an equality, and it reflects that the present
geometry is based on an essentially one--dimensional projective structure.

\begin{lemma}[Cauchy--Schwarz equality in DA geometry]\label{lem:CS-equality}
For any points $A,B$, one always has
\[
 |\langle OA,OB\rangle_{\mathcal P}|
 = |OA|_{\mathcal P}\,|OB|_{\mathcal P}.
\]
\end{lemma}

\begin{proof}
In the standard model, we have
$|OA|_{\mathcal P}=|x_A|$,
$|OB|_{\mathcal P}=|x_B|$, and
$|AB|_{\mathcal P}=|x_A-x_B|$.
Hence
\[
\langle OA,OB\rangle_{\mathcal P}
= \tfrac12\!\left(|x_A|^2+|x_B|^2-|x_B-x_A|^2\right)
= x_A x_B.
\]
Therefore
\[
|\langle OA,OB\rangle_{\mathcal P}|
= |x_A x_B|
= |OA|_{\mathcal P}\,|OB|_{\mathcal P}.
\]
\end{proof}

\begin{remark}[Sign structure of angles]\label{rem:sign-structure}
Since the DA inner product coincides with $x_Ax_B$,
the sign in the angle structure is uniquely determined by the projective direction,
that is, by the sign of the $x$--coordinate.
In particular, among the interior angles of a DA triangle,
exactly one is a negative difference angle.
\end{remark}

\begin{proposition}[Sign structure of DA triangles]\label{prop:two-positive-one-negative}
Among the interior angles $\Pangle A$, $\Pangle B$, and $\Pangle C$
of a DA triangle $\Ptri ABC$,
there exists exactly one negative difference angle.
\end{proposition}

\begin{proof}
Assume that $A,B,C$ satisfy $x_A<x_B<x_C$.
The sign of a difference angle is determined by the sign of its denominator, and
\[
(x_B-x_A)(x_C-x_A)>0,\qquad
(x_A-x_B)(x_C-x_B)<0,\qquad
(x_A-x_C)(x_B-x_C)>0.
\]
Hence the unique negative difference angle is $\Pangle B$.
\end{proof}

\begin{figure}[htbp]
  \centering
  \begin{subfigure}{0.3\textwidth}
    \centering
    \includegraphics[scale=0.9]{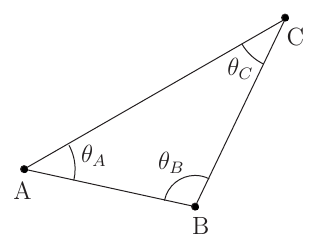}
    \subcaption{}
    \label{fig:cyclic-inner-prod1}
  \end{subfigure}
  \begin{subfigure}{0.3\textwidth}
    \centering
    \includegraphics[scale=0.9]{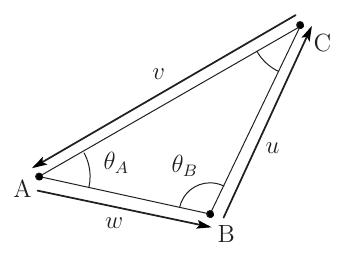}
    \subcaption{}
    \label{fig:cyclic-inner-prod2}
  \end{subfigure}
  \begin{subfigure}{0.3\textwidth}
    \centering
    \includegraphics[scale=0.9]{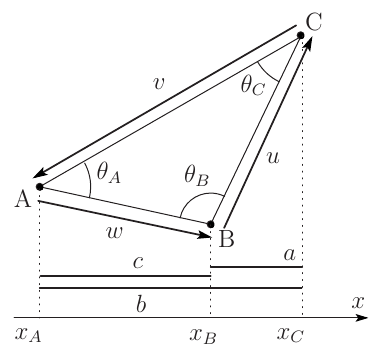}
    \subcaption{}
    \label{fig:cyclic-inner-prod3}
  \end{subfigure}
  \caption{(a) A DA triangle.
(b) Cyclic inner product identity for edge vectors.
(c) Positive cyclic identity in the normalized configuration.}
\end{figure}

\begin{lemma}[Cyclic inner product identity for edge vectors]\label{lem:cyclic-inner}
Let $u=\overrightarrow{BC}$, $v=\overrightarrow{CA}$, and $w=\overrightarrow{AB}$.
Then $u+v+w=0$, and for any (pseudo--)inner product one has
\[
\langle u,v\rangle_{\mathcal P}+\langle v,w\rangle_{\mathcal P}+\langle w,u\rangle_{\mathcal P}
=-\tfrac12\bigl(|u|_{\mathcal P}^2+|v|_{\mathcal P}^2+|w|_{\mathcal P}^2\bigr).
\]
\end{lemma}

\begin{proof}
Expanding $0=\langle u+v+w,\,u+v+w\rangle_{\mathcal P}$, we obtain
\[
|u|_{\mathcal P}^2+|v|_{\mathcal P}^2+|w|_{\mathcal P}^2
+2\bigl(\langle u,v\rangle_{\mathcal P}+\langle v,w\rangle_{\mathcal P}+\langle w,u\rangle_{\mathcal P}\bigr)=0.
\]
\end{proof}

This lemma yields an important identity that strengthens the motivation for
$\sinp\theta$ introduced later.

\begin{proposition}[Positive cyclic identity]\label{prop:positive-cyclic}
Let $\Ptri ABC$ be a DA triangle, and set
$a=|BC|_{\mathcal P}$, $b=|CA|_{\mathcal P}$, and $c=|AB|_{\mathcal P}$.
Then
\[
a^2+b^2+c^2
=2\sqrt{a^2b^2+b^2c^2+c^2a^2}.
\]
\end{proposition}

\begin{proof}
We square both sides of the cyclic inner product identity
(\Cref{lem:cyclic-inner}) and use the sign structure of DA triangles.

Assume $x_A < x_B < x_C$.
By \Cref{prop:two-positive-one-negative}, we have $\theta_B<0$.
Hence the relevant inner products satisfy
\[
\langle \overrightarrow{BC},\overrightarrow{CA}\rangle_{\mathcal P}=-ab,\qquad
\langle \overrightarrow{CA},\overrightarrow{AB}\rangle_{\mathcal P}=-bc,\qquad
\langle \overrightarrow{AB},\overrightarrow{BC}\rangle_{\mathcal P}=ca.
\]

Moreover, since the difference--angle norm is based on a one--dimensional projective structure,
the triangle inequality is always an equality, and thus
\[
a+c=b.
\]
This reflects the linear nature peculiar to DA geometry.

Substituting these into \Cref{lem:cyclic-inner}, we obtain
\[
-2(-ab-bc+ca)=a^2+b^2+c^2 .
\]
Squaring both sides and simplifying yields
\[
4\bigl(a^2b^2+b^2c^2+c^2a^2+2abc(b-a-c)\bigr)
=(a^2+b^2+c^2)^2 .
\]
Using $a+c=b$, we have $2abc(b-a-c)=0$, and hence
\[
(a^2+b^2+c^2)^2 = 4(a^2b^2+b^2c^2+c^2a^2).
\]
Taking square roots gives the claim.
\end{proof}

\begin{proposition}[Isoptics of the DA inner product are always singular lines]\label{prop:isoptic}
For a fixed point $A$ and a constant $c\in\R$, define
\[
\mathcal{I}(A;c)
:=\{\,X\mid \langle OA,OX\rangle_{\mathcal P}=c\,\}.
\]
If $x_A\neq0$, then $\mathcal{I}(A;c)$ is the line (a singular line) given by $x_X=c/x_A$.
\end{proposition}

\begin{proof}
By \Cref{lem:CS-equality}, we have $\langle OA,OX\rangle_{\mathcal P}=x_A x_X=c$.
\end{proof}

\begin{remark}[Contrast with Euclidean isoptics]
In the Euclidean inner product, an isoptic curve (a constant-angle locus)
is in general not a line; it is a conic defined by $\angle(OA,OX)=\text{const.}$
In DA geometry, since the inner product degenerates to a one--dimensional projective structure,
every isoptic coincides with a singular line.
\end{remark}

\par\smallskip

Finally, in Euclidean geometry the parallelogram law was generalized by M.\,Stewart
to what is now known as Stewart's theorem.
In DA geometry, the same lineage is inherited through the inner product:
our inner product reproduces a Stewart-type identity.

\begin{corollary}[Difference--angle Stewart's theorem]\label{cor:DA-Stewart}
Let $\Ptri ABC$ be a DA triangle, and let $S$ be a point on $AB$.
If $|AS|_{\mathcal P}=p$ and $|SB|_{\mathcal P}=q$, then
\[
  (q|CA|_{\mathcal P})^2 + (p|CB|_{\mathcal P})^2
    = |AB|_{\mathcal P}\Bigl(|CS|_{\mathcal P}^2 + pq\Bigr).
\]
\end{corollary}

\begin{proof}
We first treat the case where $B$ is the negative difference-angle vertex in $\Ptri ABC$,
namely, $x_A<x_B<x_C$.
Since
\[
\vec{CS}=\frac{q}{p+q}\vec{CA}+\frac{p}{p+q}\vec{CB},
\]
we have
\begin{align*}
|CS|_{\mathcal P}^2 + pq
&=\frac{q^2}{(p+q)^2}|CA|_{\mathcal P}^2 +\frac{p^2}{(p+q)^2}|CB|_{\mathcal P}^2
   +2\frac{pq}{(p+q)^2} \langle \vec{CA},\vec{CB}\rangle_{\mathcal P}+pq\\
&=\frac{q^2}{(p+q)^2}|CA|_{\mathcal P}^2 +\frac{p^2}{(p+q)^2}|CB|_{\mathcal P}^2
   +2\frac{pq}{(p+q)^2}|CA|_{\mathcal P}|CB|_{\mathcal P}+pq\\
&=\frac{1}{p+q}\left\{\frac{q^2}{p+q}|CA|_{\mathcal P}^2 +\frac{p^2}{p+q}|CB|_{\mathcal P}^2
   +\frac{2pq}{p+q}|CA|_{\mathcal P}|CB|_{\mathcal P}+pq(p+q)\right\}\\
&=\frac{1}{p+q}\left\{q|CA|_{\mathcal P}^2+p|CB|_{\mathcal P}^2
   -\frac{pq}{p+q}\bigl(|CA|_{\mathcal P}-|CB|_{\mathcal P}\bigr)^2+pq(p+q)\right\}.
\end{align*}
Since $x_A<x_B<x_C$, the definition of the difference--angle norm yields
\[
|CA|_{\mathcal P}-|CB|_{\mathcal P}=|AB|_{\mathcal P}.
\]
Therefore,
\[
|CS|_{\mathcal P}^2 + pq
=\frac{1}{p+q}\left\{q|CA|_{\mathcal P}^2+p|CB|_{\mathcal P}^2
-\frac{pq}{p+q}|AB|_{\mathcal P}^2+pq(p+q)\right\}.
\]
Since $|AB|_{\mathcal P}=p+q$, we obtain
\[
|CS|_{\mathcal P}^2 + pq
=\frac{1}{p+q}\left\{q|CA|_{\mathcal P}^2+p|CB|_{\mathcal P}^2\right\},
\]
and multiplying both sides by $p+q$ gives the desired identity.

If $C$ is the negative difference-angle vertex instead, then $\vec{CA}$ and $\vec{CB}$
have opposite directions, and using the triangle equality
$|CA|_{\mathcal P}+|CB|_{\mathcal P}=|AB|_{\mathcal P}$,
the same conclusion follows in the same manner.
\end{proof}

\begin{remark}[Geometric consistency]
This corollary shows that the DA inner product, constructed solely from the angle axioms,
naturally reproduces classical metric relations.
In particular, the Stewart-type identity appearing here is an automatic consequence
of the intrinsic structure of the axiomatic system.
\end{remark}

\begin{remark}[An alternative geometric derivation]\label{rem:stewart-geo}
The Stewart-type identity follows directly from the bilinearity of the DA inner product,
but the same relation can also be obtained by combining
the parabolic Ptolemy theorem and the parabolic power theorem.

More precisely, introduce the point $D$ where the Ceva line $CS$ through $S\in AB$
meets the circumparabola. Using similarity relations among
$\triangle ADS$, $\triangle BDS$ and $\triangle ACS$, $\triangle BCS$,
together with
\[
 |AS|_{\mathcal P}\,|BS|_{\mathcal P}=|CS|_{\mathcal P}\,|SD|_{\mathcal P}
\]
(the parabolic power),
one can derive Stewart's identity without algebraic computations
(see Appendix for details).
\end{remark}


\section{Parabolic trigonometric functions}\label{sec:parabolic-trig}

\begin{quote}
In \emph{DA geometry}, if one defines the ``cosine'' in the usual way by
\(
\cos\theta=\dfrac{\langle u,v\rangle_{\mathcal P}}{|u|_{\mathcal P}|v|_{\mathcal P}},
\)
then the Cauchy--Schwarz inequality always becomes an equality, so that this
quantity carries essentially no Euclidean-type information: $\cos\theta$ can
take only the values of absolute value $1$.

Nevertheless, once one takes into account quantities such as the area of a DA
triangle, it is still necessary to introduce an appropriate notion of
trigonometric functions in DA geometry.

Accordingly, we denote the trigonometric functions in DA geometry by
$\cosp\theta$ and $\sinp\theta$, and summarize the requirements they should
satisfy as follows:
\begin{itemize}
\item $\sinp\theta$ should be an odd function, and $\cosp^{2}\theta$ should not depend on the orientation of the difference angle.
\item $\sinp\theta$ should represent a quantity related to area.
\item In analogy with the Euclidean identity
      $\langle u,v\rangle = |u||v|\cos\theta$,
      one should be able to write
      $\langle u,v\rangle_{\mathcal P}
      =|u|_{\mathcal P}|v|_{\mathcal P}\,\cosp\theta$.
\end{itemize}
From this perspective, together with the viewpoint developed in the next
section---namely, that DA geometry arises as a \emph{parabolic linear
degeneration limit}---it is natural to regard the linear approximation
\[
e^{i\theta}=\cos\theta + i\sin\theta
\qquad\leadsto\qquad
e^{i\theta} \approx 1 + i\theta
\quad(\theta\to 0)
\]
as the relevant guiding principle.
On the other hand, in DA geometry a straight angle corresponds to $0$, and it is
natural to treat the angle quantity linearly both when $\theta$ is near $0$ and
when it is near $\pi$.
In other words, both $\theta$ and $(\theta-\pi)$ play the role of ``small''
quantities.

Based on this guiding principle, we define parabolic trigonometric functions and
justify the definition by deriving analogues of classical statements in Euclidean
geometry.
\end{quote}

\subsection{Definitions and basic structure}

\begin{definition}[Parabolic trigonometric functions]\label{def:parabolic-trig}
For a difference angle $\theta(\neq 0)$, define
\[
\cosp\theta := \operatorname{sgn}(\theta),\qquad
\sinp\theta := \theta.
\]
We call $\cosp$ the \emph{parabolic cosine} and $\sinp$ the \emph{parabolic sine}.
Then $\sinp$ is an odd function, $\cosp$ is an odd function, and $\cosp^2\theta=1$.
Here $(\cosp\theta)^2$ is understood as $\cosp^2\theta$ in the usual sense.
\end{definition}

\begin{remark}[The value at $\theta=0$]
For the value of $\cosp 0$, it may be natural to set $\cosp 0=0$.
However, since this value does not appear in the arguments of the present paper,
we leave it unspecified for the time being.
\end{remark}

\begin{remark}[Embedding into a parabola parameter]\label{rem:embedding-parabola}
To a difference angle $\theta$ we associate the point
\[
\iota(\theta):=\left(\frac{\theta}{\kappa},\frac{\theta^2}{\kappa}\right)\in\{y=\kappa x^2\}.
\]
Then the slope of the segment $OP$ equals $\theta$.
\end{remark}

From the definition, $\cosp\theta$ depends only on whether $\theta$ is a positive
or a negative difference angle (hence it is determined as $\pm 1$), while
$\sinp\theta$ is fully linearized. In particular, the following properties hold.

\begin{proposition}\label{prop:parabolic-trig-basic}
The functions $\cosp\theta$ and $\sinp\theta$ satisfy:
\begin{itemize}
\item $\cosp^2\theta =1$ (in particular, $\cosp\theta\in\{\pm1\}$),
\item for $k\in \Z^{*}$, $\cosp(k\theta)=\cosp(\theta)$,
\item for $k\in \Z^{*}$, $\sinp(k\theta)=k\,\sinp(\theta)$.
\end{itemize}
\end{proposition}

\begin{figure}[H]
  \centering
  \begin{minipage}{0.60\textwidth}
    \centering
    \includegraphics[width=\linewidth]{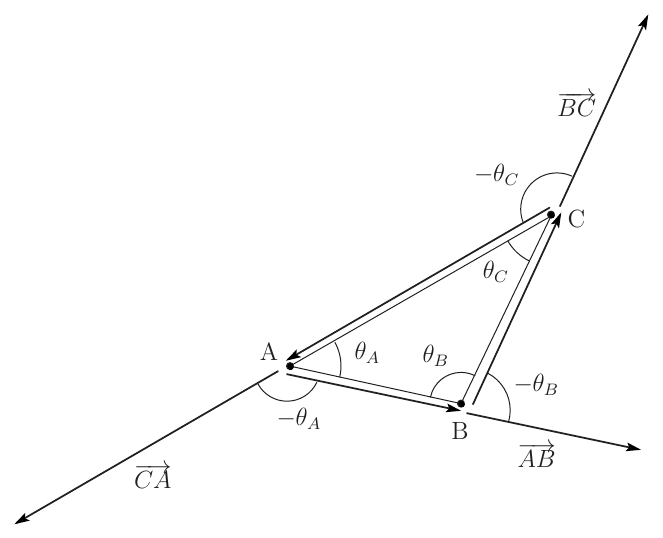}
  \end{minipage}
  \caption{Configuration for the proof of the first cosine law in DA geometry.}
  \label{fig:first-cosp-formula}
\end{figure}

\begin{proposition}[First cosine law]\label{prop:first-cosine-law}
Let $\Ptri ABC$ be a DA triangle, and set
$a=|BC|_{\mathcal{P}}$, $b=|CA|_{\mathcal{P}}$, $c=|AB|_{\mathcal{P}}$,
$\Pangle BAC=\theta_A$, $\Pangle CBA=\theta_B$, and $\Pangle ACB=\theta_C$.
Then
\[
a=b\cosp \theta_C+c\cosp \theta_B,\qquad
b=c\cosp \theta_A+a\cosp \theta_C,\qquad
c=a\cosp \theta_B+b\cosp \theta_A.
\]
\end{proposition}

\begin{proof}
As in the Euclidean proof of the first cosine law, take the DA inner product of
the relation $\vec{AB}+\vec{BC}+\vec{CA}=0$ with $\vec{AB}$ (and similarly with
$\vec{BC}$ and $\vec{CA}$).
From
\[
\vec{AB}=-\vec{BC}-\vec{CA}
\]
we obtain
\[
\langle \vec{AB},\vec{AB}\rangle_{\mathcal P}
=-\langle \vec{BC},\vec{AB}\rangle_{\mathcal P}
 -\langle \vec{CA},\vec{AB}\rangle_{\mathcal P}.
\]
Hence,
\[
c^2
=- ac\,\cosp(-\theta_B)- bc\,\cosp(-\theta_A)
\iff
c=a\cosp \theta_B+ b\cosp\theta_A,
\]
where we used $\cosp(-\theta)=-\cosp(\theta)$.
The other identities follow in the same way.
\end{proof}

\begin{remark}[The second cosine law is absorbed]\label{rem:second-cosine-absorbed}
As in the Euclidean proof of the second cosine law, for a DA triangle $\Ptri ABC$
we have
\[
\langle \vec{AB},\vec{AB}\rangle_{\mathcal P}
=\langle -\vec{BC}-\vec{CA},\, -\vec{BC}-\vec{CA}\rangle_{\mathcal P},
\]
which yields the ``second cosine law''
\[
a^2=b^2+c^2-2bc\cosp \theta_A.
\]
However, under the triangle equality in DA geometry, this identity reduces to
the first cosine law itself.
\end{remark}

\begin{remark}\label{rem:asa-sas-comment}
The first cosine law may be viewed as reflecting the influence of ASA-type
congruence information, while the second cosine law is typically tied to
SAS-type information.
In the present setting, the second cosine law is absorbed into the first, which
suggests that ASA-type information plays a more essential role in DA geometry.

On the other hand, if one formally considers ``linearized trigonometric
functions'' such as
\[
\cos^\ast\theta = 1,\qquad \sin^\ast\theta = \theta,
\]
then the triangle equality and cosine laws degenerate and do not provide a rich
geometric structure for general triangles.
By contrast, the definitions of $\cosp$ and $\sinp$ in
\Cref{def:parabolic-trig} are compatible not only with the DA inner product and
the DA norm, but also naturally lead to several distinctive identities for
parabolic triangles.
We postpone the details to the next section, where it will become clear that
these linearized parabolic trigonometric functions interact remarkably well
with the fundamental quadratic structure underlying DA geometry.
\end{remark}

\subsection{Justification of parabolic trigonometric functions}

\begin{quote}
There are various parametrizations of pairs $(x,y)$ satisfying the relation
$y=\kappa x^2$, and one could in principle attach the names $\cosp$ and $\sinp$
to any of them.
On the other hand, the functions $\cosp$ and $\sinp$ defined in
\Cref{def:parabolic-trig} are directly tied to the one-dimensional projective
structure of the parabola as the basic figure, and they most straightforwardly
reflect the linearity peculiar to DA geometry.
In this subsection, we give one concrete justification of the parabolic
trigonometric functions defined in the previous subsection, via the Brocard
angle.

The main theorem treated here was not discussed for parabolas even in the work
of Weiss--Odehnal; it is obtained by reinterpreting their considerations within
the framework of DA geometry.
\end{quote}

\begin{figure}[htbp]
  \centering
  \begin{minipage}{0.40\textwidth}
    \centering
    \includegraphics[width=\linewidth]{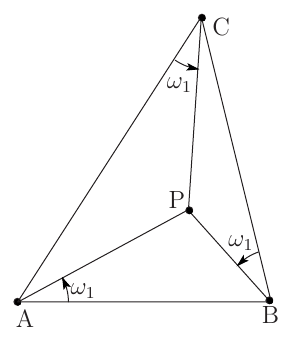}
    \subcaption{}
    \label{fig:brocard-in-Euclid}
  \end{minipage}
  \begin{minipage}{0.40\textwidth}
    \centering
    \includegraphics[width=\linewidth]{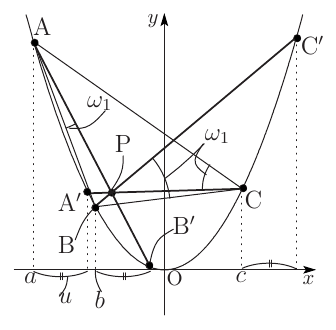}
    \subcaption{}
    \label{fig:brocard-in-DA}
  \end{minipage}
 \caption{Brocard's theorem in Euclidean and DA geometries.
(a) The Brocard point and Brocard angle in Euclidean geometry.
(b) The Brocard point and Brocard angle in DA geometry.}
\end{figure}

We begin by recalling the following theorem in Euclidean geometry.
Since it is classical, we omit the proof.

\begin{proposition}[Brocard's theorem]
Let $ABC$ be a triangle in the Euclidean plane.
There exists a unique interior point $P_1$ such that, for the Euclidean angle
$\angle$,
\[
\angle P_1BC=\angle P_1CA=\angle P_1AB=\omega_1.
\]
The point $P_1$ is called the \emph{first Brocard point}, and $\omega_1$ is
called the \emph{first Brocard angle}.
Similarly, there exists a unique interior point $P_2$ such that
\[
\angle P_2CB=\angle P_2AC=\angle P_2BA=\omega_2.
\]
The point $P_2$ is called the \emph{second Brocard point}, and $\omega_2$ is
called the \emph{second Brocard angle}.
\end{proposition}

The Brocard angle $\omega(=\omega_1,\omega_2)$ admits expressions in terms of
side lengths and area.
Let $S$ be the area of $\triangle ABC$, and let $|BC|=a$, $|CA|=b$, $|AB|=c$.
Then
\begin{align*}
\tan \omega&= \frac{4S}{a^2+b^2+c^2},\\
\sin \omega&=\frac{2S}{\sqrt{a^2b^2+b^2c^2+c^2a^2}},\\
\cot \omega&= \cot \angle A+\cot \angle B+\cot \angle C.
\end{align*}
We will show that Brocard's theorem also holds in DA geometry, and then deduce
that the expression in terms of $\cot$ is preserved.

\begin{maintheorem}[DA Brocard theorem (existence and uniqueness)]\label{mthm:da-brocard}
Let $\Ptri ABC$ be a DA triangle whose circumparabola is $y=\kappa x^2$ with
$\kappa>0$.
Then there exists a unique point $P_1$ such that
\[
\Pangle P_1BC=\Pangle P_1CA=\Pangle P_1AB=\omega_1.
\]
Moreover, there exists a unique point $P_2$ such that
\[
\Pangle P_2CB=\Pangle P_2AC=\Pangle P_2BA=\omega_2.
\]
\end{maintheorem}

\begin{proof}
To treat the three equal-angle conditions, write $x_A=a$, $x_B=b$, $x_C=c$
with $a<b<c$.
Fix $\omega\neq 0$, and consider a point $X$ satisfying
\[
\Pangle XBC=\Pangle XCA=\omega,
\]
and a point $Y$ satisfying
\[
\Pangle YCA=\Pangle YAB=\omega.
\]
We determine $\omega$ so that $X=Y$.

Put $u=\dfrac{\omega}{\kappa}$, and choose points $A',B',C'$ on the parabola
$y=\kappa x^2$ by
\[
x_{A'}=a+u,\quad x_{B'}=b+u,\quad x_{C'}=c+u.
\]
By the definition of the difference angle, the condition that $X$ is the
intersection point of the lines $BC'$ and $CA'$ is equivalent to
\[
\Pangle XBC=\Pangle XCA=\omega,
\]
and similarly, the condition that $Y$ is the intersection point of the lines
$CA'$ and $AB'$ is equivalent to
\[
\Pangle YCA=\Pangle YAB=\omega.
\]

The equations of the lines $BC'$, $CA'$, and $AB'$ are
\begin{align*}
BC':\quad y&=\kappa(b+c+u)x-\kappa b(c+u),\\
CA':\quad y&=\kappa(c+a+u)x-\kappa c(a+u),\\
AB':\quad y&=\kappa(a+b+u)x-\kappa a(b+u).
\end{align*}
Hence, from
\[
\kappa(b+c+u)x_X-\kappa b(c+u)=\kappa(c+a+u)x_X-\kappa c(a+u)
\]
we obtain
\[
x_X=c+\frac{b-c}{b-a}\,u.
\]
Similarly, from
\[
\kappa(c+a+u)x_Y-\kappa c(a+u)=\kappa(a+b+u)x_Y-\kappa a(b+u)
\]
we obtain
\[
x_Y=a+\frac{c-a}{c-b}\,u.
\]
Therefore, the condition $x_X=x_Y$ (equivalently, $X=Y$) is
\begin{align}
c+\frac{b-c}{b-a}u=a+\frac{c-a}{c-b}u
&\iff
\frac{(c-a)(b-a)-(b-c)(c-b)}{(b-a)(c-b)}u=c-a \notag\\
&\iff
u=\frac{(b-a)(c-b)(c-a)}{(c-a)(b-a)-(b-c)(c-b)} \notag\\
&\iff
u=\frac{(b-a)(c-b)(c-a)}{a^2+b^2+c^2-ab-bc-ca}.
\label{eq:brocard-u}
\end{align}
Thus \eqref{eq:brocard-u} determines $u$ uniquely (hence $\omega=\kappa u$ is
unique).
For this $u$, the lines $BC'$, $CA'$, and $AB'$ indeed determine points $X$ and
$Y$, and since $x_X=x_Y$ we have $X=Y$.
Denoting this point by $P_1$ proves the first assertion.

For $\omega_2$, it suffices to repeat the above construction with $u$ replaced
by $-u$.
In particular, $\omega_1$ and $\omega_2$ are isogonal conjugates in this sense.

Finally, note that
\[
a^2+b^2+c^2-ab-bc-ca
=\tfrac12\bigl((a-b)^2+(b-c)^2+(c-a)^2\bigr)>0.
\]
\end{proof}

\begin{definition}
The point $P_1$ in \Cref{mthm:da-brocard} is called the \emph{first Brocard
point}, and $\omega_1$ is called the \emph{first Brocard angle}.
Similarly, $P_2$ is called the \emph{second Brocard point}, and $\omega_2$ is
called the \emph{second Brocard angle}.
\end{definition}

As in Euclidean geometry, the Brocard angle $\omega=\omega_1,\omega_2$ in DA
geometry is also related to area.
We next state the required facts about the area of a DA triangle.
As in the previous chapter, we postpone a deeper discussion to a subsequent
paper and present only what is needed here.
The proof follows immediately from the cross product of
$\overrightarrow{AB}$ and $\overrightarrow{AC}$.

\begin{lemma}\label{lem:da-triangle-area}
Let $\Ptri ABC$ be a DA triangle whose circumparabola is $y=\kappa x^2$ with
$\kappa>0$.
Assume $x_A=a$, $x_B=b$, $x_C=c$ with $a<b<c$.
If $S_{\mathcal P}$ denotes the area of $\Ptri ABC$, then
\[
S_{\mathcal P}=\frac{\kappa}{2}(b-a)(c-b)(c-a).
\]
\end{lemma}

Next we define a parabolic tangent.

\begin{definition}[Parabolic tangent]\label{def:parabolic-tangent}
For the point $P=\iota(\theta)$ given by the embedding
$\iota(\theta)=(\theta/\kappa,\theta^2/\kappa)$, define
\[
\tanp\theta:=\Slp(OP)=\theta\ (\theta\neq 0),\qquad \tanp0:=0.
\]
\end{definition}

\begin{remark}\label{rem:tanp-not-ratio}
Since $\cosp\theta\in\{\pm1\}$ and $\cosp\theta$ and $\sinp\theta$ have the same
sign, the ratio $\sinp\theta/\cosp\theta$ does not reflect geometric
information (here, a slope).
For this reason, $\tanp$ is defined not as a ratio but as the slope induced by
the embedding $\iota$.
With this definition, $\tanp\theta=\theta$ is naturally defined for all
$\theta\in\R$ and functions as a primary first-order quantity in DA geometry.
\end{remark}

Combining this definition of $\tanp$ with \Cref{lem:da-triangle-area} and
\eqref{eq:brocard-u}, we obtain the following relations between the Brocard
angle and the side lengths.

\begin{proposition}[Parabolic trigonometric expressions for the Brocard angle]\label{prop:parabolic-brocard-trig}
Let $S_{\mathcal P}$ be the area of $\Ptri ABC$.
Then the angle $\omega=\omega_1$ in \Cref{mthm:da-brocard} satisfies
\begin{align*}
\tanp\omega&=\frac{4S_{\mathcal{P}}}{{|AB|_{\mathcal{P}}}^2+{|BC|_{\mathcal{P}}}^2+{|CA|_{\mathcal{P}}}^2},\\
\sinp\omega&=\frac{2S_{\mathcal{P}}}{\sqrt{{|CA|_{\mathcal{P}}}^2{|AB|_{\mathcal{P}}}^2+{|BC|_{\mathcal{P}}}^2{|AB|_{\mathcal{P}}}^2+{|BC|_{\mathcal{P}}}^2{|CA|_{\mathcal{P}}}^2}}.
\end{align*}
\end{proposition}

\begin{proof}
Since
$S_{\mathcal{P}}=\dfrac{\kappa}{2}(b-a)(c-b)(c-a)$,
$|AB|_{\mathcal{P}}=b-a$,
$|BC|_{\mathcal{P}}=c-b$,
and $|CA|_{\mathcal{P}}=c-a$,
writing $\omega=\kappa u$ we have
\[
\omega=\kappa u
=\frac{\kappa (b-a)(c-b)(c-a)}{a^2+b^2+c^2-ab-bc-ca}
=\frac{2\kappa (b-a)(c-b)(c-a)}{(b-a)^2+(c-b)^2+(a-c)^2}.
\]
Hence
\[
\tanp\omega=\omega
=\frac{\kappa (b-a)(c-b)(c-a)}{a^2+b^2+c^2-ab-bc-ca}
=\frac{2\kappa (b-a)(c-b)(c-a)}{(b-a)^2+(c-b)^2+(a-c)^2}.
\]
On the other hand, by \Cref{prop:positive-cyclic},
\[
(b-a)^2+(c-b)^2+(a-c)^2
=2\sqrt{(c-a)^2(b-a)^2+(c-b)^2(b-a)^2+(c-b)^2(c-a)^2}.
\]
Therefore,
\[
\sinp\omega
=\omega
=\frac{4S_{\mathcal P}}{2\sqrt{(c-a)^2(b-a)^2+(c-b)^2(b-a)^2+(c-b)^2(c-a)^2}},
\]
which matches the claimed two formulas.
For $\omega_2$, the same statements hold with the opposite sign, since
$|\omega_2|=|\omega_1|$.
\end{proof}

\begin{remark}
In \eqref{eq:brocard-u} we have
\[
u=\frac{(b-a)(c-b)(c-a)}{a^2+b^2+c^2-ab-bc-ca}>0,
\]
since both the numerator and the denominator are positive.
Hence
\[
\omega=\kappa u>0\qquad(\kappa>0).
\]
In particular, the first Brocard angle $\omega_1$ is always positive, whereas
the second Brocard angle $\omega_2$ is always negative in the sense of
difference angles.
\end{remark}

\begin{proposition}\label{prop:cotp-brocard}
For a difference angle $\theta(\neq 0)$ define
\[
\cotp\theta=\frac{1}{\tanp\theta}=\frac{1}{\theta}.
\]
Then the angle $\omega=\omega_1$ in \Cref{mthm:da-brocard} satisfies
\[
\cotp\omega=\cotp\theta_A+\cotp\theta_B+\cotp\theta_C.
\]
\end{proposition}

\begin{proof}
Write $\omega_1=\omega=\kappa u$ with $u>0$.
Then $\cotp\omega=\dfrac{1}{\omega}=\dfrac{1}{\kappa u}$, and
\[
u=\frac{(b-a)(c-b)(c-a)}{a^2+b^2+c^2-ab-bc-ca}
=\frac{2(b-a)(c-b)(c-a)}{(b-a)^2+(c-b)^2+(a-c)^2}.
\]
Thus
\[
\begin{aligned}
\frac{1}{\kappa u}
&=\frac{(b-a)^2+(c-b)^2+(c-a)^2}{2\kappa(b-a)(c-b)(c-a)}\\
&=\frac{1}{2\kappa}\left(
\frac{b-a}{(c-b)(c-a)}
+\frac{c-b}{(b-a)(c-a)}
+\frac{c-a}{(b-a)(c-b)}
\right)\\
&=\frac{1}{2\kappa}\left(
\frac{1}{c-b}-\frac{1}{c-a}
+\frac{1}{b-a}-\frac{1}{c-a}
+\frac{1}{b-a}+\frac{1}{c-b}
\right)\\
&=\frac{1}{\kappa(c-b)}-\frac{1}{\kappa(c-a)}+\frac{1}{\kappa(b-a)}.
\end{aligned}
\]
Since $a<b<c$, by \Cref{prop:two-positive-one-negative} we have $\theta_B<0$, and
\[
\theta_A=\kappa(c-b),\quad \theta_C=\kappa(b-a),\quad \theta_B=\kappa(a-c).
\]
Therefore
\[
\frac{1}{\kappa(c-b)}-\frac{1}{\kappa(c-a)}+\frac{1}{\kappa(b-a)}
=\frac{1}{\theta_A}+\frac{1}{\theta_B}+\frac{1}{\theta_C}.
\]
Hence
\[
\frac{1}{\omega}=\frac{1}{\theta_A}+\frac{1}{\theta_B}+\frac{1}{\theta_C}
\iff
\cotp \omega= \cotp \theta_A+\cotp \theta_B+\cotp \theta_C.
\]
\end{proof}

\begin{remark}
In both \Cref{prop:parabolic-brocard-trig} and \Cref{prop:cotp-brocard}, the
corresponding statements for $\omega=\omega_2$ follow with the opposite sign,
since $\omega_2=-\omega_1$.
\end{remark}

\subsection{Alternating product}\label{sec:alt-prod-and-content}

\begin{quote}
The DA inner product was defined as a bilinear form derived from the
parallelogram identity, and it is given by
\[
\langle u,v\rangle_{\mathcal P}=x_u x_v.
\]
Formally one may write
\(
\langle u,v\rangle_{\mathcal P}
=|u|_{\mathcal P}|v|_{\mathcal P}\,\cosp\theta,
\)
so it is natural to introduce an alternating quantity that corresponds to the
Euclidean exterior product $|u||v|\sin\theta$.
We postpone the discussion of a sine law to a subsequent paper, since it goes
beyond the scope of the present section.
\end{quote}

\begin{definition}[Alternating product associated with the DA inner product]\label{def:alt-prod}
Using the parabolic trigonometric functions $\cosp$ and $\sinp$, define
\[
\Pi_{\mathcal P}(u,v)
:= |u|_{\mathcal P}\,|v|_{\mathcal P}\,\sinp\theta
\qquad (\theta=\Pangle (u,v)).
\]
We call $\Pi_{\mathcal P}(u,v)$ the \emph{alternating product} associated with
the DA inner product.
In Euclidean geometry it coincides with $|u||v|\sin\theta$, and in DA geometry
it provides a prototype of an area quantity.
(If $u$ or $v$ is in the singular direction, we agree that
$\Pi_{\mathcal P}(u,v)=0$.)
\end{definition}

\begin{proposition}[Alternating property]
For any vectors $u,v$, we have
\[
\Pi_{\mathcal P}(u,v)=-\Pi_{\mathcal P}(v,u).
\]
\end{proposition}

\begin{proof}
This follows from the antisymmetry of the difference angle,
\[
\Pangle(v,u)=-\Pangle(u,v),
\]
together with the oddness of $\sinp$.
\end{proof}

\begin{remark}
The alternating property of $\Pi_{\mathcal P}$ reflects the orientation
structure inherent in DA geometry: reversing the difference-angle direction
reverses the orientation of the associated area-type quantity.
\end{remark}

\subsection{Summary}
The parabolic trigonometric functions introduced in this section simultaneously satisfy:
\begin{enumerate}
\item an analytic description of the inner-product structure derived from the
      parallelogram identity,
\item a unified formulation of the DA norm and area information,
\item the realization of a regular limit within the Cayley--Klein family.
\end{enumerate}


\section{Parabolic limits in Cayley--Klein geometry and the DA structure}\label{sec:CK-parabolic-limit}

In this section we make explicit how the defining formulas of angles and
distances (norms) in Cayley--Klein geometry pass to the basic structures of
DA geometry through a degeneration of the absolute conic (the parabolic limit).

\subsection{Laguerre's formula and a general expression for the Cayley--Klein angle}

\begin{figure}[htbp]
  \centering
  \begin{minipage}{0.45\textwidth}
    \centering
    \includegraphics[width=\linewidth]{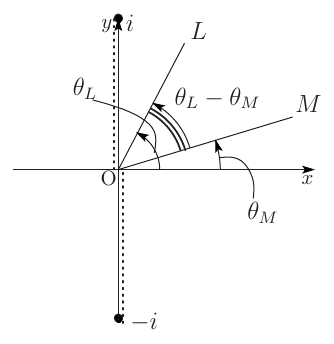}
  \end{minipage}
 \caption{
Geometric configuration for Laguerre's formula.
The angle between two lines $L$ and $M$ is represented as the difference
$\theta_L-\theta_M$ measured with respect to the isotropic directions.
}
\label{fig:laguerre}
\end{figure}

\begin{quote}
In this section, we introduce the Cayley--Klein distance (hereafter CK distance)
and the Cayley--Klein angle (hereafter CK angle).
To present a standard form of the CK angle,
we first define the isotropic lines determined by an absolute conic.

It is well known that the CK angle can be defined in terms of
the cross ratio of directions of lines.
In order to clarify the motivation for this construction,
we return to Laguerre's angle formula.
Although the original source of Laguerre's formula could not be located,
we take as a guiding reference the exposition of the Cayley--Klein metric
\footnote{
URL: \url{https://en.wikipedia.org/wiki/Cayley-Klein_metric}
},
which attributes to Laguerre the idea that an angle can be expressed
as the logarithm of a cross ratio.
We therefore begin by deriving Laguerre's formula for the Euclidean angle.

Based on this viewpoint, we then define the CK angle using the isotropic directions
determined by the absolute conic,
and verify that the resulting angle quantity satisfies the axioms
\Cref{ax:A1}--\Cref{ax:A5}.
Finally, we investigate the parabolic degeneration limit,
which leads naturally to the difference--angle structure.
\end{quote}

In this subsection we write the cross ratio of four points $a,b,c,d$ as
$\Cr(a,b;c,d)$, namely
\[
\Cr(a,b;c,d):=\frac{(a-c)/(b-c)}{(a-d)/(b-d)}.
\]

We regard directions of lines as points of the complex projective line
$\mathbb{CP}^1$.
Angles are ordinary Euclidean angles, and for a line $L$ through the origin we
write its oriented direction by the angle $\theta_L$ with the $x$-axis:
\[
  L:\ \text{direction angle }\theta_L\quad\Bigl(-\frac{\pi}{2}<\theta_L\le\frac{\pi}{2}\Bigr).
\]
Similarly, for a line $M$ with direction angle $\theta_M$, we define the
(oriented) Euclidean angle between $L$ and $M$ by
\[
  \varphi := \theta_L - \theta_M,
\]
viewed in $(-\pi,\pi]$ if necessary.

For simplicity we write the slopes of $L$ and $M$ as
\[
m_L := \tan\theta_L,\quad m_M := \tan\theta_M.
\]

Laguerre's idea may be understood as expressing the angle between two lines
$L$ and $M$ as a cross ratio of the four directions determined by $L$, $M$ and
two reference lines $t_1,t_2$.
If directions are treated as points of $\mathbb{CP}^1$, the cross ratio is a
projective invariant, and this becomes an angle transformation formula.

\begin{theorem}[Laguerre's angle formula]
In the above setting, the Euclidean angle $\varphi=\theta_L-\theta_M$ between
$L$ and $M$ satisfies
\[
  \Cr(m_L,m_M;i,-i) = e^{2i\varphi}
\]
when we choose the direction parameters $m_{t_1}=i$ and $m_{t_2}=-i$
corresponding to the circular points at infinity\footnote{%
In the complex projective line $\mathbb{CP}^1$, these correspond to the
circular points at infinity.}.
Consequently, after fixing a branch,
\[
\varphi=\frac{1}{2i}\log \Cr(m_L,m_M;i,-i).
\]
\end{theorem}

\begin{proof}
We assume Euler's formula $e^{i\theta}=\cos\theta+i\sin\theta$.
Choosing $m_{t_1}=i$ and $m_{t_2}=-i$ and using $m_L=\tan\theta_L$ and
$m_M=\tan\theta_M$, we compute
\begin{align*}
\Cr(m_L,m_M; m_{t_1}, m_{t_2})
&=\frac{\tan \theta_L-i}{\tan \theta_M-i}\Big/\frac{\tan \theta_L+i}{\tan \theta_M+i}\\
&=\frac{(\sin\theta_L-i\cos\theta_L)(\sin\theta_M+i\cos\theta_M)}
        {(\sin\theta_M-i\cos\theta_M)(\sin\theta_L+i\cos\theta_L)}\\
&=\frac{(\cos\theta_L+i\sin\theta_L)(\cos(-\theta_M)+i\sin(-\theta_M))}
        {(\cos\theta_M+i\sin\theta_M)(\cos(-\theta_L)+i\sin(-\theta_L))}\\
&=\frac{e^{i\theta_L}e^{-i\theta_M}}{e^{i\theta_M}e^{-i\theta_L}}\\
&=\frac{e^{i(\theta_L-\theta_M)}}{e^{i(\theta_M-\theta_L)}}\\
&=e^{2i(\theta_L-\theta_M)}.
\end{align*}
Therefore,
\[
2i(\theta_L-\theta_M)=\log \Cr(m_L,m_M; i,-i)
\quad\iff\quad
\theta_L-\theta_M=\frac{1}{2i}\log \Cr(m_L,m_M; i,-i).
\]
\end{proof}

In Cayley--Klein geometry one fixes an absolute conic in the projective plane,
\[
\mathcal Q:\quad Ax^2+By^2+Cz^2+Dxy+Eyz+Fzx=0.
\]
Within this framework, Laguerre's formula provides the viewpoint of measuring
slopes of lines via a cross ratio.
In the Euclidean case, the reference directions $i$ and $-i$ can be interpreted
as the tangency directions at the circular points in the complex projective
plane.
Generalizing this idea, one may define an angle by taking the cross ratio of
directions with respect to the isotropic directions determined by $\mathcal Q$.
This is the Cayley--Klein angle.

We first specify the two distinguished isotropic directions determined by the
absolute conic.

\begin{definition}[Isotropic lines]\label{def:isotropic-lines}
Fix an absolute conic $\mathcal Q$ in the projective plane.
Let $I_{t,1}$ and $I_{t,2}$ denote the two tangency directions associated with
$\mathcal Q$.
Any line parallel to either of these directions is called an \emph{isotropic
line} of the corresponding geometry\footnote{%
When the absolute is the unit circle in the Euclidean plane, isotropic lines
correspond to the lines tangent to the circular points at infinity in the
complex projective plane and provide the unit directions for angles.}.
\end{definition}

\begin{figure}[htbp]
  \centering
  \begin{minipage}{0.40\textwidth}
    \centering
    \includegraphics[width=\linewidth]{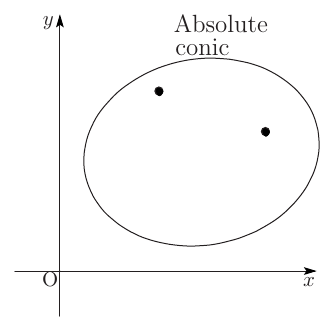}
    \subcaption{}
    \label{fig:ck-geometry}
  \end{minipage}
  \begin{minipage}{0.40\textwidth}
    \centering
    \includegraphics[width=\linewidth]{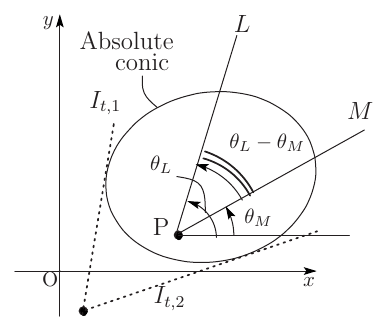}
    \subcaption{}
    \label{fig:ck-angle-def}
  \end{minipage}
 \caption{
Cayley--Klein geometry and its angle structure.
(a) Absolute conic defining the geometry.
(b) The Cayley--Klein angle defined by the cross ratio of directions with
respect to the isotropic lines $I_{t,1}, I_{t,2}$.
}
\end{figure}

\begin{definition}[Cayley--Klein angle]\label{def:ck-angle}
Fix an absolute conic $\mathcal Q$ in the projective plane, and let
$I_{t,1},I_{t,2}$ be the two isotropic directions as in
\Cref{def:isotropic-lines}.
Choose a point $P$ outside $\mathcal Q$, and let $L$ and $M$ be two lines
through $P$ that are not isotropic.

Choose an affine chart and represent directions of lines by the slope parameter
$m\in\R\cup\{\infty\}$.
Writing $m_L,m_M,m_{t_1},m_{t_2}$ for the slopes of $L,M,I_{t,1},I_{t,2}$, the
cross ratio
\[
\Cr(m_L,m_M;m_{t_1},m_{t_2})
=\frac{(m_L-m_{t_1})(m_M-m_{t_2})}{(m_L-m_{t_2})(m_M-m_{t_1})}
\]
is invariant under projective transformations preserving $\mathcal Q$.

For a constant $\lambda\neq 0$, define
\[
\angle_{CK}(L,M)
:=\lambda\,\log\Cr(m_L,m_M;m_{t_1},m_{t_2}),
\]
where $\log$ is taken as the principal branch.
We call this quantity the \emph{Cayley--Klein angle} (CK angle) associated with
$\mathcal Q$.
We will also restrict to the region where $\Cr$ is positive (i.e., where the
relevant directions lie in the same connected component).
\end{definition}

\begin{theorem}\label{thm:ck-angle-axioms}
The angle quantity $\angle_{CK}$ defined in \Cref{def:ck-angle} satisfies
\Cref{ax:A1}--\Cref{ax:A5} after restricting to a suitable domain avoiding the
isotropic lines $I_{t,1},I_{t,2}$.
\end{theorem}

\begin{proof}
We verify the axioms.

\medskip\noindent
\Cref{ax:A1} (opposite angle / order reversal).
For two lines $L,M$ through $P$,
\begin{align*}
\angle_{CK}(M,L)
&=\lambda\log\Cr(m_M,m_L;m_{t_1},m_{t_2})\\
&=\lambda\log\frac{(m_M-m_{t_1})(m_L-m_{t_2})}{(m_M-m_{t_2})(m_L-m_{t_1})}\\
&=-\lambda\log\frac{(m_L-m_{t_1})(m_M-m_{t_2})}{(m_L-m_{t_2})(m_M-m_{t_1})}\\
&=-\lambda\log\Cr(m_L,m_M;m_{t_1},m_{t_2})\\
&=-\angle_{CK}(L,M).
\end{align*}

\medskip\noindent
\Cref{ax:A2} (additivity).
For three oriented lines $\ell_1,\ell_2,\ell_3$ through $P$,
\begin{align*}
\angle_{CK}(\ell_1,\ell_2)+\angle_{CK}(\ell_2,\ell_3)
&=\lambda\log\Cr(m_{\ell_1},m_{\ell_2};m_{t_1},m_{t_2})
 +\lambda\log\Cr(m_{\ell_2},m_{\ell_3};m_{t_1},m_{t_2})\\
&=\lambda\log\frac{(m_{\ell_1}-m_{t_1})(m_{\ell_2}-m_{t_2})}{(m_{\ell_1}-m_{t_2})(m_{\ell_2}-m_{t_1})}
   +\lambda\log\frac{(m_{\ell_2}-m_{t_1})(m_{\ell_3}-m_{t_2})}{(m_{\ell_2}-m_{t_2})(m_{\ell_3}-m_{t_1})}\\
&=\lambda\log\frac{(m_{\ell_1}-m_{t_1})(m_{\ell_3}-m_{t_2})}{(m_{\ell_1}-m_{t_2})(m_{\ell_3}-m_{t_1})}\\
&=\lambda\log\Cr(m_{\ell_1},m_{\ell_3};m_{t_1},m_{t_2})\\
&=\angle_{CK}(\ell_1,\ell_3).
\end{align*}

\medskip\noindent
\Cref{ax:A3} (vanishing).
If $m_{\ell_1}=m_{\ell_2}$, then
\begin{align*}
\angle_{CK}(\ell_1,\ell_2)
&=\lambda\log\Cr(m_{\ell_1},m_{\ell_1};m_{t_1},m_{t_2})\\
&=\lambda\log\frac{(m_{\ell_1}-m_{t_1})(m_{\ell_1}-m_{t_2})}{(m_{\ell_1}-m_{t_2})(m_{\ell_1}-m_{t_1})}\\
&=\lambda\log 1=0.
\end{align*}
Conversely, if $\angle_{CK}(\ell_1,\ell_2)=0$, then
\[
\Cr(m_{\ell_1},m_{\ell_2};m_{t_1},m_{t_2})=1,
\]
hence
\[
(m_{\ell_1}-m_{t_1})(m_{\ell_2}-m_{t_2})
=(m_{\ell_1}-m_{t_2})(m_{\ell_2}-m_{t_1}),
\]
which implies
\[
(m_{\ell_1}-m_{\ell_2})(m_{t_1}-m_{t_2})=0.
\]
Since the two isotropic directions are distinct, $m_{t_1}\neq m_{t_2}$, so
$m_{\ell_1}=m_{\ell_2}$ follows.

\medskip\noindent
\Cref{ax:A4} (scaling invariance along rays).
Let $\overrightarrow{PX}$ denote an oriented ray from $P$.
If $A'\in\overrightarrow{PA}$ and $B'\in\overrightarrow{PB}$, then $PA$ and
$PA'$ are the same line, and similarly $PB$ and $PB'$ are the same line.
Thus
\[
m_{PA}=m_{PA'},\qquad m_{PB}=m_{PB'}.
\]
Therefore,
\begin{align*}
\angle_{CK}(PA,PB)
&=\lambda\log\Cr(m_{PA},m_{PB};m_{t_1},m_{t_2})\\
&=\lambda\log\Cr(m_{PA'},m_{PB'};m_{t_1},m_{t_2})\\
&=\angle_{CK}(PA',PB').
\end{align*}

\begin{remark}
A4 asserts that the angle quantity depends only on the \emph{direction} of the
lines and not on the choice of points along the rays (i.e., not on scale).
\end{remark}

\medskip\noindent
\Cref{ax:A5} (continuous divisibility).
(i) (existence of an angle bisector).
Let $\ell_1,\ell_2$ be two lines through $P$, with slopes
$m_{\ell_1},m_{\ell_2}$.
A bisector line $M$ satisfies
\[
\angle_{CK}(\ell_1,M)=\angle_{CK}(M,\ell_2).
\tag{$\ast$}
\]
By the definition of the CK angle, this is equivalent to
\[
\Cr(m_{\ell_1},m_M;m_{t_1},m_{t_2})
=
\Cr(m_M,m_{\ell_2};m_{t_1},m_{t_2}).
\tag{$\ast\ast$}
\]
Both sides are rational functions of $m_M$.
As $m_M$ varies within any connected component of
$\R\cup\{\infty\}\setminus\{m_{t_1},m_{t_2}\}$, the left-hand side is monotone
in one direction while the right-hand side is monotone in the opposite
direction (or vice versa).
Hence, by the intermediate value theorem, there exists a unique $m_M$ solving
\((\ast\ast)\).

\smallskip\noindent
(ii) (continuous subdivision).
Fix $P$ and let $\mathcal R_P$ be the topological space of oriented rays through
$P$.
Let $S_P\subset\mathcal R_P$ be the singular set consisting of rays in isotropic
directions, and put $D_P=\mathcal R_P\setminus S_P$.
Using the slope $m$ as a local coordinate, $D_P$ is a disjoint union of open
intervals.
For each connected component $C\subset D_P$, fix a reference ray $r\in C$.
Then the map
\[
C\ni s \longmapsto \angle_{CK}(r,s)
   =\lambda\log\Cr(m_r,m_s;m_{t_1},m_{t_2})
\]
is a rational function of $m_s$, hence real-analytic away from the singular
values $m_{t_1},m_{t_2}$.
In particular it is continuous on $C$, and \Cref{ax:A5}(ii) follows.
\end{proof}

\begin{remark}
In a suitable affine coordinate system, let $m_L$ denote the slope of a line
$L$.
For four lines $L_1,L_2,L_3,L_4$, we write
\[
\Cr(L_1,L_2;L_3,L_4)
:= \Cr(m_{L_1},m_{L_2};m_{L_3},m_{L_4}).
\]
\end{remark}

\subsection{Linear Degeneration Limit of the CK Angle}

\begin{figure}[htbp]
  \centering
  \begin{minipage}{0.40\textwidth}
    \centering
    \includegraphics[width=\linewidth]{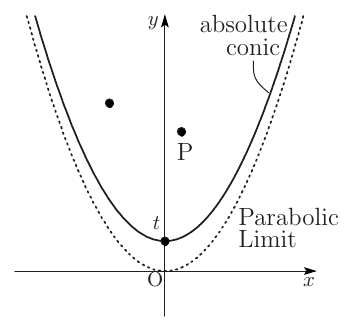}
    \subcaption{}
    \label{fig:fig-parabolic-abso}
  \end{minipage}
  \begin{minipage}{0.40\textwidth}
    \centering
    \includegraphics[width=\linewidth]{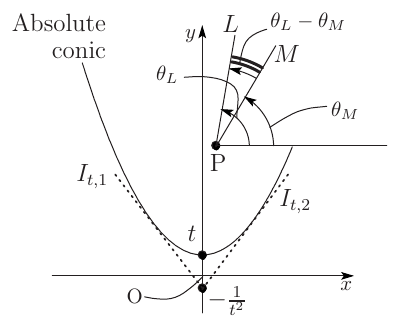}
    \subcaption{}
    \label{fig:parabolic-angle}
  \end{minipage}
 \caption{
(a) Parabolic degeneration of the absolute conic $\mathcal Q_t$ as $t\to0^+$, converging to the limit parabola $y=\kappa x^2$.
(b) Isotropic lines $I_{t,1}, I_{t,2}$ associated with the absolute conic $\mathcal Q_t$,
and the Cayley--Klein angle $\theta_L-\theta_M$ defined by the cross ratio of directions.
In the parabolic limit $t\to0^+$, the isotropic lines degenerate into a single direction.
}

\end{figure}
\begin{quote}
The quantity defined in \Cref{def:ck-angle} has been recognized as a valid notion of angle.
In this subsection, we consider a one--parameter family of absolute conics $Q_t$ given by parabolas,
and derive the parabolic difference angle $\Pangle$
by taking the first--order approximation of the logarithm of the cross ratio defining the CK angle
as $t\to0$.
Through this limiting process, DA geometry is obtained as a consistent linear geometric model
with the parabola as the absolute figure.
\end{quote}

In Cayley--Klein geometry, once an absolute conic $Q$ is fixed,
its tangents determine the isotropic lines,
and these two distinguished directions provide the reference for angle measurement.
More precisely,

\begin{itemize}
  \item when the isotropic lines define two distinct real directions,
        the geometry is of hyperbolic type;
  \item when the isotropic lines form a complex conjugate pair
        and do not appear as real lines,
        the geometry is of elliptic type.
\end{itemize}

This corresponds to the classification of Cayley--Klein geometries
according to the signature of the absolute conic
(elliptic, hyperbolic, and parabolic types).

In this section, we fix $\kappa>0$ and consider a family of parabolas
\[
\mathcal Q_t:\quad y=\kappa x^2+t
\]
as absolute conics.
For $t>0$, the tangents are complex conjugate (elliptic type),
whereas for $t<0$ there exist two distinct real tangents (hyperbolic type).

To derive the difference angle, the choice of isotropic lines is essential.
In particular, we consider the isotropic lines drawn from the point
$\bigl(0,-\frac{1}{t^2}\bigr)$ $(t\neq0)$ to $Q_t$.

The slopes of these isotropic lines are obtained by imposing the double--root condition
on the equation $\kappa x^2+t=mx-\frac{1}{t^2}$, which yields
\[
m=\pm 2\sqrt{\kappa\!\left(t+\frac{1}{t^2}\right)}.
\]
As $t\to0$, we have
\[
|s(t)|:=2\sqrt{\kappa\!\left(t+\frac{1}{t^2}\right)}\longrightarrow\infty.
\]

Thus the slopes of the two isotropic lines diverge,
and both directions converge to the axis direction $x=0$.
Since in DA geometry the angle with respect to the singular direction is defined to be zero,
the linear degeneration limit of the CK angle is necessarily taken along this axial direction.

More precisely, the two isotropic lines
\[
I_{t,1}:\ y=s(t)x+\cdots,\qquad
I_{t,2}:\ y=-s(t)x+\cdots
\]
collapse, in the parabolic limit $t\to0$, to a single isotropic direction
\[
I_{t,1},I_{t,2}\;\longrightarrow\; x=0.
\]
This collapse of isotropic lines provides the geometric basis
for the linearization of the logarithmic CK angle in the parabolic limit.

We now introduce the following notation.
For a line $L$ with slope $m_L\neq0$,
and for four lines $L_1,L_2,L_3,L_4$, we define
\[
\Cr^\vee(L_1,L_2;L_3,L_4)
:= \Cr\!\left(\frac{1}{m_{L_1}},\frac{1}{m_{L_2}};
               \frac{1}{m_{L_3}},\frac{1}{m_{L_4}}\right).
\]
We consider in particular the cross ratio
\[
\Cr^\vee(\ell_1,\ell_2;I_{t,1},I_{t,2}),
\]
whose logarithmic limit will be the object of our analysis.
Hereafter, $\Cr^\vee$ is defined for $m_{L_i}\neq0$,
and cases involving horizontal lines will be treated by continuous extension
in a later remark.

Although the definition requires the two isotropic lines $I_{t,1},I_{t,2}$ to be distinct,
in the parabolic limit $s(t)\to\infty$ their reciprocal slopes satisfy
\[
\frac{1}{m_{t_1}}=\frac{1}{s(t)},\qquad
\frac{1}{m_{t_2}}=-\frac{1}{s(t)},
\]
and hence approach each other.
Consequently, the cross ratio admits an expansion of the form
\[
\Cr^\vee(\ell_1,\ell_2;I_{t,1},I_{t,2})
=1+c_1(m_1-m_2)\,u(t)+O\!\left(u(t)^2\right),
\]
where $c_1$ is a constant and $u(t)$ is a small parameter.
Therefore,
\[
\log\Cr^\vee \approx c_1(m_1-m_2)\,u(t),
\]
and the logarithm becomes linearized.

Fix a point $P$ and consider two rays $\ell_1,\ell_2$ emanating from $P$
with slopes $m_1,m_2$.
We define the parabolic CK angle by
\[
\angle_t(\ell_1,\ell_2)
:=\log\Cr^\vee(\ell_1,\ell_2;I_{t,1},I_{t,2}).
\]

\begin{lemma}\label{lem:log-expand}
For $|x|<1$, the expansion $\log(1+x)=x+O(x^2)$ holds.
\end{lemma}

\begin{lemma}[First--order expansion of the CK angle]\label{lem:first-order-CK}
Let the slopes of the isotropic lines be
\[
m_{t_1}=s(t)=\frac{1}{u(t)},\qquad
m_{t_2}=-s(t)=-\frac{1}{u(t)},
\]
where $u(t)\to0$ corresponds to the parabolic limit $t\to0$.
For slopes $m_1,m_2\neq0$ of $\ell_1,\ell_2$, we have
\[
\log\Cr^\vee(\ell_1,\ell_2;I_{t,1},I_{t,2})
=-2u(t)(m_1-m_2)+O\!\left(u(t)^3\right).
\]
\end{lemma}

Here $u(t)=1/s(t)$ is a small parameter tending to zero
in the parabolic limit $t\to0$,
and all subsequent expansions are taken with respect to this parameter.

\begin{proof}
Let $u:=u(t)$.
Then
\begin{align*}
\Cr^\vee(\ell_1,\ell_2;I_{t,1},I_{t,2})
&=\Cr\!\left(\frac{1}{m_1},\frac{1}{m_2};u,-u\right)\\
&=\frac{(1-m_1u)(1+m_2u)}{(1+m_1u)(1-m_2u)}.
\end{align*}
Taking logarithms and expanding yields
\begin{align*}
\log\Cr^\vee
&=\log(1-m_1u)-\log(1+m_1u)
 +\log(1+m_2u)-\log(1-m_2u)\\
&=-2m_1u+2m_2u+O(u^3)\\
&=-2(m_1-m_2)u+O(u^3).
\end{align*}
\end{proof}

In the parabolic limit, the variation of the cross ratio collapses to first order in
\[
u(t)=\frac{1}{s(t)}.
\]
To obtain a finite limiting angle, a first--order normalization is therefore required.
We introduce a normalization function $\alpha(t)$ satisfying
\[
\alpha(t)\sim -2u(t).
\]

\begin{definition}[Normalization of the parabolic angle]\label{def:parabolic-angle}
Let $u(t)=1/s(t)$ and let $\alpha(t)$ satisfy $\alpha(t)\sim -2u(t)$.
For two rays $\ell_1,\ell_2$ with slopes $m_1,m_2\neq0$,
the parabolic angle is defined by
\[
\Pangle(\ell_1,\ell_2)
:=\lim_{t\to0}\frac{1}{\alpha(t)}
\log\Cr^\vee(\ell_1,\ell_2;I_{t,1},I_{t,2}).
\]
\end{definition}

\begin{maintheorem}[Linearization of the CK angle]\label{thm:CK-linearization}
Let $s(t)\to\infty$ (equivalently $u(t)\to0$),
and let $\alpha(t)\sim -2u(t)$.
For slopes $m_1,m_2\neq0$, the parabolic angle defined above satisfies
\[
\Pangle(\ell_1,\ell_2)=m_1-m_2.
\]
\end{maintheorem}

\begin{proof}
By definition and \Cref{lem:first-order-CK},
\[
\frac{1}{\alpha(t)}
\log\Cr^\vee(\ell_1,\ell_2;I_{t,1},I_{t,2})
=(m_1-m_2)+O\!\left(u(t)^2\right).
\]
Taking the limit $t\to0$ yields the result.
\end{proof}

\begin{remark}[Including horizontal lines]
Since the right--hand side of \Cref{lem:first-order-CK}
is analytic in $m_1,m_2$,
the formula extends continuously to the case $m_1=0$ or $m_2=0$.
Thus the same expression
\[
\Pangle(\ell_1,\ell_2)=m_1-m_2
\]
holds whenever horizontal lines are involved.
\end{remark}

In this way, we confirm that the parabolic linear degeneration limit
of the CK angle converges to the difference angle.
The choice of isotropic lines employed here is not merely a computational convenience,
but serves to clarify geometrically how the isotropic directions collapse
into a single direction in the parabolic limit.
Indeed, performing the same analysis using the point $(0,-1/t)$
leads to no essential difference in the resulting first--order angle.

This reflects the fact that the elliptic and hyperbolic distinctions
in Cayley--Klein geometry become inessential in the parabolic limit.
From this viewpoint, the difference angle can be interpreted as a finite,
first--order quantity arising on the boundary of the CK framework.
Accordingly, DA geometry is positioned as a degeneration model
distinct from the classical trichotomy
(elliptic, hyperbolic, and Euclidean) of Cayley--Klein geometries.

\begin{remark}[Role of DA geometry]
Within this degeneration framework,
the difference angle may be interpreted as capturing
an ``extremely microscopic'' notion of angle.
On the other hand, in view of the parabolic trigonometric identity
$\cosp\theta=\pm1$,
it may also be regarded as a quantity that appears microscopic
when the geometry is viewed from a macroscopic perspective.
\end{remark}

\begin{remark}[A degenerate angle quantity contrasting with the difference angle]
\label{rem:dual-diff-angle}
In the above degenerate limit, if one constructs a Cayley--Klein geometry
by choosing the isotropic lines to pass through the origin,
then a similar first--order approximation yields only a linearization of the form
\[
\log\Cr(\ell_1,\ell_2; I_{t,1},I_{t,2})
= d_1\!\left(\frac{1}{m_1}-\frac{1}{m_2}\right)u(t)
  +O\!\left(u(t)^3\right),
\qquad (d_1 \text{ is a constant}),
\]
and no other leading behavior appears.
That is, the resulting angle quantity is not given by the difference of slopes
$m_1-m_2$, but rather by the difference of their reciprocals
$\frac{1}{m_1}-\frac{1}{m_2}$.
Since this angle quantity arises as a degenerate limit of the Cayley--Klein angle,
it naturally carries an axiomatic system of angles satisfying
\Cref{ax:A1}--\Cref{ax:A5}.
In this sense, it can be interpreted as having a character contrasting with that of
the difference angle.
Indeed, the isoptic curves associated with this angle degenerate linearly and appear
as straight lines, in clear contrast to the parabolic isoptics arising in
difference--angle geometry.

Nevertheless, the geometry based on this angle degenerates essentially into a linear
structure from a geometric point of view, and therefore does not give rise to the rich
tangential structure or triangle geometry observed in difference--angle geometry.
For this reason, the present paper focuses on difference--angle geometry, which possesses
more substantial geometric content, and does not pursue a detailed study of this
degenerate geometry.
\end{remark}

\subsection{Degeneration Limit of the CK Distance}
\begin{figure}[htbp]
  \centering
  \begin{minipage}{0.40\textwidth}
    \centering
    \includegraphics[width=\linewidth]{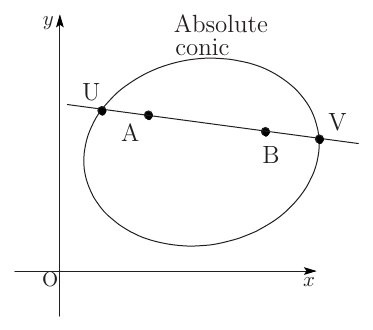}
    \subcaption{}
    \label{fig:ck-dist}
  \end{minipage}
  \begin{minipage}{0.40\textwidth}
    \centering
    \includegraphics[width=\linewidth]{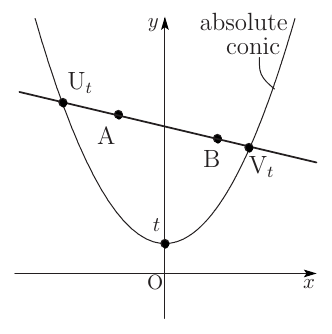}
    \subcaption{}
    \label{fig:parabolic-dist}
  \end{minipage}
 \caption{
Cayley--Klein distance and its parabolic degeneration.
(a) Definition of the Cayley--Klein distance between $A$ and $B$ via the intersection points
$U,V$ with the absolute conic.
(b) Degenerating case where the absolute conic is given by $y=\kappa x^{2}+t$,
approaching the parabolic limit as $t\to0^{+}$.
}

\end{figure}

Here we again take the projective reference line $\ell$ to be the $x$--axis
and the projection direction $d$ to be the $y$--axis direction.
As in the previous subsection, we fix $\kappa>0$ and consider the family of absolute conics
\[
\mathcal Q_t:\quad y=\kappa x^2+t\qquad (t>0).
\]
Then:

\begin{itemize}
  \item For $t>0$, the conic $\mathcal Q_t$ can be brought, by an affine change of coordinates,
        into a Cayley--Klein standard form (elliptic or hyperbolic type).
  \item As $t\to0^{+}$, the conic $\mathcal Q_t$ degenerates to the parabola $y=\kappa x^2$,
        yielding a parabolic limit.
  \item We restrict projective transformations to those preserving the pair $(\ell,d)$.
  \item For any line $AB$, the intersection points $U_t,V_t=AB\cap\mathcal Q_t$
        converge to points on the limit parabola, and the logarithm of the associated cross ratio
        admits a first--order (linear) expansion.
\end{itemize}

Let $\mathcal Q_t$ be an absolute conic in the (real or complex) projective plane.
For two points $A,B$, let $U_t,V_t$ be the intersection points of the line $AB$ with $\mathcal Q_t$.
The Cayley--Klein distance is defined by the logarithmic cross ratio
\[
d_t(A,B)
= \frac{1}{2}
\left|
\log\frac{\overline{AU_t}\cdot \overline{BV_t}}
{\overline{AV_t}\cdot \overline{BU_t}}
\right|.
\]
Here $\overline{PQ}$ denotes the affine length ratio along the projective line $PQ$.

\begin{proposition}[CK distance]\label{prop:CK-distance}
Let $\mathcal C:\,y=\kappa x^2$ be the parabola and take two distinct points
$A=(x_A,\kappa x_A^2)$ and $B=(x_B,\kappa x_B^2)$ on $\mathcal C$ ($x_A\neq x_B$).
For $\mathcal Q_t:\,y=\kappa x^2+t$, let $U_t,V_t$ be the intersection points of $AB$ with $\mathcal Q_t$.
Then the Cayley--Klein distance satisfies
\[
d_t(A,B)
=\left|\log\!\left(\frac{b-\Delta}{b+\Delta}\right)\right|,
\qquad
b:=x_B-x_A,\quad
\Delta:=\sqrt{\,b^2-\frac{4t}{\kappa}\,}.
\]
\end{proposition}

\begin{proof}
The chord $AB$ has equation $y=\kappa(x_A+x_B)x-\kappa x_Ax_B$.
Substituting into $\mathcal Q_t$ yields
\[
x^2-(x_A+x_B)x+x_Ax_B+\frac{t}{\kappa}=0,
\]
whose two roots are
\[
x_{U_t}=\frac{x_A+x_B-\Delta}{2},\qquad
x_{V_t}=\frac{x_A+x_B+\Delta}{2}.
\]
Hence
\[
\frac{\overline{AU_t}}{\overline{BU_t}}
=\frac{x_{U_t}-x_A}{x_B-x_{U_t}}
=\frac{b-\Delta}{b+\Delta},\qquad
\frac{\overline{AV_t}}{\overline{BV_t}}
=\frac{b+\Delta}{b-\Delta}.
\]
Therefore,
\[
d_t(A,B)
=\frac{1}{2}\left|
\log\!\left(
\frac{\overline{AU_t}}{\overline{AV_t}}
\cdot
\frac{\overline{BV_t}}{\overline{BU_t}}
\right)\right|
=\left|\log\!\left(\frac{b-\Delta}{b+\Delta}\right)\right|.
\]
\end{proof}

\begin{lemma}\label{lem:binom-expansions}
For $|x|<1$,
\[
(1-x)^{1/2}=1-\frac{x}{2}+O(x^2),
\qquad
\log(1+x)=x+O(x^2).
\]
\end{lemma}

Applying this to the parameter $b$ in \Cref{prop:CK-distance},
we obtain a first--order approximation with respect to the small parameter $t$.

\begin{corollary}[Specialization]\label{cor:binom-expansions-special}
Let $\varepsilon:=\frac{4t}{\kappa b^2}$. Then as $t\to0$,
\[
\sqrt{1-\varepsilon}
=1-\frac{\varepsilon}{2}+O(\varepsilon^2)
=1-\frac{2t}{\kappa b^2}+O(t^2).
\]
\end{corollary}

\begin{proposition}[Small deviation]\label{prop:small-deviation}
Let $\mathcal C:\,y=\kappa x^2$ be the circumparabola and take
$A=(x_A,\kappa x_A^2)$ and $B=(x_B,\kappa x_B^2)$ on $\mathcal C$ ($x_A\neq x_B$).
For each $t>0$, consider $\mathcal Q_t:\,y=\kappa x^2+t$ and let $U_t,V_t$ be the intersection points of $AB$ with $\mathcal Q_t$.
Then there exists $\alpha_t=O(t)$ such that
\[
\frac{\overline{AU_t}}{\overline{BU_t}}
=\frac{t}{\kappa b^2}\bigl(1+\alpha_t\bigr),
\qquad
b:=x_B-x_A.
\]
\end{proposition}

\begin{proof}
By \Cref{prop:CK-distance},
\(
\frac{\overline{AU_t}}{\overline{BU_t}}=\frac{b-\Delta}{b+\Delta}
\)
with
$\Delta=b\sqrt{1-\tfrac{4t}{\kappa b^2}}$.
Using \Cref{cor:binom-expansions-special}, we obtain
\[
\frac{b-\Delta}{b+\Delta}
=\frac{1-\sqrt{1-\frac{4t}{\kappa b^2}}}
{1+\sqrt{1-\frac{4t}{\kappa b^2}}}
=\frac{\kappa b^2}{4t}\!\left(1-\sqrt{1-\frac{4t}{\kappa b^2}}\right)^{\!2}
=\frac{t}{\kappa b^2}\bigl(1+O(t)\bigr).
\]
Taking $\alpha_t:=O(t)$ yields the claim.
\end{proof}

\begin{proposition}[Linearization of the CK distance]
Let $A=(x_A,\kappa x_A^2)$ and $B=(x_B,\kappa x_B^2)$ be points on $\mathcal Q_t:\,y=\kappa x^2+t$.
Let $U_t,V_t$ be the intersection points of the chord $AB$ with $\mathcal Q_t$.
Then the CK distance satisfies
\[
d_t(A,B)
=\frac{1}{2}\left|(\alpha_t-\beta_t)+o(|\alpha_t-\beta_t|)\right|,
\]
where $\alpha_t,\beta_t=O(t)$ are the first--order deviations arising from the intersection points of $AB$ with $\mathcal Q_t$.
\end{proposition}

\begin{proof}
As in \Cref{prop:small-deviation}, there exists $\beta_t=O(t)$ such that
\[
\frac{\overline{AV_t}}{\overline{BV_t}}
=\frac{t}{\kappa b^2}\bigl(1+\beta_t\bigr),
\qquad
b:=x_B-x_A.
\]
Hence, by the definition of the CK distance,
\[
d_t(A,B)
= \frac{1}{2}
\left|
\log(1+\alpha_t)-\log(1+\beta_t)
\right|.
\]
Applying \Cref{lem:binom-expansions} gives
\[
\log(1+\alpha_t)-\log(1+\beta_t)
=(\alpha_t-\beta_t)+o(|\alpha_t-\beta_t|),
\qquad t\to0.
\]
Substituting this into the previous expression yields the claim.
\end{proof}

\begin{proposition}[Parabolic limit of the CK distance]
Under the parabolic limit $t\to0^{+}$, the CK distance associated with $\mathcal Q_t$
satisfies
\[
\lim_{t\to0^{+}} d_t(A,B)
=|x_B-x_A|=:|AB|_{\mathcal P}.
\]
\end{proposition}

\begin{proof}
The quantity $\alpha_t-\beta_t$ represents the first--order deviation along the direction of the chord $AB$.
In the affine coordinates $(x,y)$ adapted to $(\ell,d)$, only the $x$--component contributes to the limit.
Accordingly, there exists a positive function $c(t)$ (depending only on the choice of normalization in the CK metric)
such that
\[
\alpha_t-\beta_t = c(t)\,(x_B-x_A)+o(t).
\]
By the preceding proposition,
\[
d_t(A,B)
=\frac{1}{2}\,|c(t)(x_B-x_A)|+o(t).
\]
Defining $|AB|_{\mathcal P}:=|x_B-x_A|$,
we may choose the overall multiplicative constant in the definition of the CK distance
(i.e.\ the normalization of the CK metric) so that $c(t)\to 2$ as $t\to0^{+}$.
With this normalization, we obtain
\[
\lim_{t\to0^{+}}d_t(A,B)
=|AB|_{\mathcal P}.
\]
\end{proof}

\begin{remark}
This provides the prototype of the norm in DA geometry.
\end{remark}

\subsection{Incompleteness of Cayley--Klein Geometry and the Role of DA Geometry}

As we have seen so far, the notions of distance and angle in Cayley--Klein geometry
are defined through a fixed absolute conic $\mathcal Q$.
In this framework, the three types---elliptic, hyperbolic, and Euclidean---are classified
according to the signature of a bilinear form $B(x,y)$ determined by the absolute.
On the basis of this classification, Klein advocated the well-known ``trichotomy'' of geometries.
However, the representative distance expression
\[
  d(x,y)
  = \arccos\!\left(
      \frac{ \lvert B(x,y) \rvert }
           { \sqrt{\,B(x,x)\,B(y,y)\,} }
    \right)
\]
is no longer directly meaningful in the parabolic degeneration,
namely, when $\mathcal Q$ degenerates by becoming tangent to the line at infinity.
Indeed, in the limit $t\to 0$, the cross-ratio formulas defining distance and angle
typically exhibit a cancellation between numerator and denominator,
so that the leading term in the logarithm disappears and the naive formula becomes indeterminate.

Klein himself already emphasized this point in
\emph{Vergleichende Betrachtungen \\"uber neuere geometrische Forschungen}
(1873; English translation available)~\cite{Klein1893ErlangenEng},
where he remarked that
\emph{``In the limiting case where the absolute conic degenerates to a parabola,
the system is not completely determined.''}
In other words, Klein clearly recognized that, in the parabolic limit,
the Cayley--Klein framework remains underdetermined.

In the present paper, we address this underdetermined parabolic limit in Klein's sense
by introducing DA geometry and by constructing, in a coherent manner,
both a distance (norm) and an angle adapted to the limiting parabola.

\begin{pmaintheorem}[Underdeterminacy of the parabolic limit in Cayley--Klein geometry]
In Cayley--Klein geometry, suppose that distance and angle are defined via a bilinear form $B(x,y)$ by
\[
  d(x,y)
  = \arccos\!\left(
      \frac{ \lvert B(x,y) \rvert }
           { \sqrt{\,B(x,x)\,B(y,y)\,} }
    \right).
\]
Then Klein's trichotomy (elliptic, Euclidean, hyperbolic) does not, by itself,
yield a complete geometric system in the parabolic degeneration.
\end{pmaintheorem}

\begin{proof}
Consider the parabolic degeneration of the absolute conic
\[
\mathcal Q_t:\ y=\kappa x^2+t \qquad (t>0).
\]
In this limit, the nonlinear structure represented by the $\arccos$-type expression
is effectively linearized, in accordance with the expansion
$\log(1+x)=x+o(x)$, after rewriting the CK definitions in logarithmic cross-ratio form.
Under this limiting procedure, the CK distance $d_t(A,B)$ and the CK angle
$\angle_t(\ell_1,\ell_2)$ converge, respectively, to
\[
d_0(A,B)=|x_B-x_A|, \qquad
\angle_0(\ell_1,\ell_2)=\Slp{\ell_1}-\Slp{\ell_2}.
\]
These limits coincide with the definitions of the DA norm and the difference angle
in DA geometry.

Therefore, by introducing DA geometry as the geometric model corresponding to the parabolic limit,
the parabolic degeneration in the Cayley--Klein framework becomes describable in a constructive manner.
\end{proof}

\begin{remark}[Significance of the parabolic geometry]
The construction of the parabolic limit given in this section extends the Cayley--Klein classification
while preserving its projective unifying viewpoint.
DA geometry provides an analytic model that explicitly supplies both distance and angle
for the parabolic degeneration that Klein left underdetermined.
In this sense, DA geometry can be regarded as a natural completion of the parabolic limit
within the broader Cayley--Klein perspective.
\end{remark}

\section*{Outlook}
In this paper, through the introduction of parabolic power,
an inner product structure, and parabolic trigonometric functions
within difference--angle geometry,
we have clarified several fundamental properties of geometries
in which angles are treated as primary quantities.
On the other hand, the present study is still restricted to the
two--dimensional setting, and the existence of analogous structures
in higher dimensions or in other degenerate limits
remains an open problem for future research.

Moreover, the difference angle appearing in this paper
and the structures contrasting with it
are expected to share formal and structural features
with the frameworks discussed in Cayley--Klein geometry
and in connection with Hilbert's Fourth Problem.
However, the purpose of the present paper is not to provide
direct solutions to these problems or a comprehensive positioning
within those broader theories.

As directions for subsequent research, we mention:
the development of area structures for difference--angle triangles;
further applications of parabolic power;
the study of circular and hyperbolic structures
within difference--angle geometry;
the investigation of group representations associated with this geometry;
and its relation to geodesic problems
within the framework of Cayley--Klein geometry.
In particular, the connection with Hilbert's Fourth Problem
will be reconsidered in a separate paper,
taking into account the structural features revealed by difference--angle geometry.
Other aspects that remain unexplored in establishing the overall framework
will likewise be addressed in future work.

Finally, this study is expected to have implications
beyond difference--angle geometry itself,
in particular for the problem of the existence of angles
in affine geometry.
Traditionally, it has been widely accepted
that no affine--invariant notion of angle can be defined
under affine transformations.
However, the system of angle axioms
\Cref{ax:A1}--\Cref{ax:A5}
suggests that this understanding may warrant reconsideration.
That is, even within affine geometry,
it may be possible to define a natural angular quantity
(an affine angle)
satisfying properties analogous to those of the CK angle,
and to construct a model of angle geometry
distinct from difference--angle geometry,
in which the corresponding isoptic curves
appear as pairs of hyperbolas.


\appendix
\section{On a Difference--Angle Geometric Proof of Stewart's Theorem}

In the main body of this paper, one of the roles of the difference--angle inner product
was illustrated by \Cref{cor:DA-Stewart}, namely the difference--angle version of Stewart's theorem,
whose proof relied on the inner product structure.
For the reader's convenience, we briefly recall the statement.

Let $\Ptri ABC$ be a difference--angle triangle, and let $S$ be a point on the side $AB$.
Set
\[
p:=|AS|_{\mathcal P},\qquad
q:=|BS|_{\mathcal P},\qquad
s:=|CS|_{\mathcal P}.
\]
Then the following identity holds:
\[
p\,|BC|_{\mathcal P}^{2}
+ q\,|CA|_{\mathcal P}^{2}
=\bigl(s^{2}+pq\bigr)\,|AB|_{\mathcal P}.
\]

In this appendix, as a representative example of a typical result in DA geometry,
we present an alternative proof of this theorem using only elementary ingredients,
namely the difference--angle version of Ptolemy's theorem and the parabolic power theorem.
The proof is computationally minimal and reflects the author's preferred approach.
It is worth noting that this method can also be transferred to the classical
Stewart's theorem in Euclidean geometry.

\begin{figure}[htbp]
  \centering
  \begin{minipage}{0.40\textwidth}
    \centering
    \includegraphics[width=\linewidth]{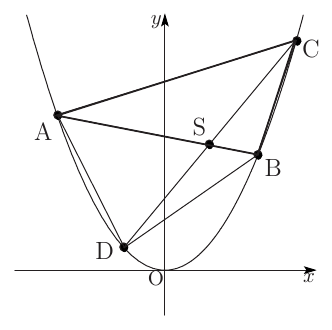}
    \subcaption{}
    \label{fig:another-proof-of-Stewart}
  \end{minipage}
  \begin{minipage}{0.40\textwidth}
    \centering
    \includegraphics[width=\linewidth]{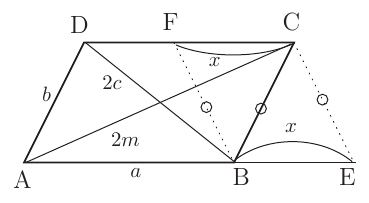}
    \subcaption{}
    \label{fig:Ptolemy-to-parallelogram}
  \end{minipage}
 \caption{Appendix figures.
(a) Geometric configuration used in the proof of the difference--angle
Stewart's theorem.
(b) Reference configuration showing the derivation of the parallelogram
theorem from Ptolemy's theorem (see footnote).}
\end{figure}

\begin{proof}
We combine the difference--angle Ptolemy theorem and the parabolic power theorem
established earlier.
Let $D$ be the intersection point of the cevian $CS$ with the circumparabola
of the difference--angle triangle $\Ptri ABC$.

Consider the difference--angle cyclic quadrilateral
$\square_{\mathcal P} ADBC$, and let $S$ be the intersection point of its diagonals
$AB$ and $CD$.
Applying the parabolic power theorem and the difference--angle Ptolemy theorem,
we derive the desired identity.

Set again
$p:=|AS|_{\mathcal P}$,
$q:=|BS|_{\mathcal P}$,
$s:=|CS|_{\mathcal P}$.
By \Cref{thm:parabolic-power}, we have
\[
|AS|_{\mathcal P}|BS|_{\mathcal P}
=|CS|_{\mathcal P}|SD|_{\mathcal P}.
\]
Moreover, from the similarity relations
$\Ptri ADS\sim_{\mathcal P}\Ptri CBS$ and
$\Ptri BDS\sim_{\mathcal P}\Ptri CAS$, it follows that
\[
|SD|_{\mathcal P}=\frac{pq}{s},\qquad
|AD|_{\mathcal P}=\frac{p}{s}|BC|_{\mathcal P},\qquad
|BD|_{\mathcal P}=\frac{q}{s}|CA|_{\mathcal P}.
\]
Therefore, by the difference--angle Ptolemy theorem,
\[
\frac{p}{s}\,|BC|_{\mathcal P}^{2}
+ \frac{q}{s}\,|CA|_{\mathcal P}^{2}
=\bigl(|CS|_{\mathcal P}+\frac{pq}{s}\bigr)|AB|_{\mathcal P}.
\]
Multiplying both sides by $s$ and rearranging yields the required equality.
\end{proof}

\begin{remark}
It is well known that the parallelogram theorem can be derived solely from
Ptolemy's theorem\footnote{
Given a parallelogram $ABCD$ with $AB=a$, $AD=b$ $(a>b)$,
choose points $E\in AB$ and $F\in CD$ such that $BF=BC=CE$.
This yields two isosceles trapezoids $AECD$ and $ABFD$.
Applying Ptolemy's theorem to both and eliminating redundant terms
gives the result.}.
The above proof illustrates that the parallelogram structure,
via the parabolic power theorem and Ptolemy's theorem,
connects succinctly to circles in Euclidean geometry
and to parabolas in DA geometry, respectively.
\end{remark}

\section{Axiom System for Angles}

For convenience, we reproduce here only the axioms concerning angles
that are required from Base~1.

\begin{angleaxiom}[Opposite angles (order reversal)]\label{ax:A1}
For any points $A,P,B$,
\[
\angle(A,P,B)=-\,\angle(B,P,A).
\]
\end{angleaxiom}

\begin{angleaxiom}[Additivity]\label{ax:A2}
If points $A,B,C$ lie on a line with $B$ between $A$ and $C$,
and if $P$ lies off this line so that the angles are continuously defined, then
\[
\angle(A,P,B)+\angle(B,P,C)=\angle(A,P,C).
\]
\end{angleaxiom}

\begin{angleaxiom}[Vanishing (definition of a straight angle)]\label{ax:A3}
If $P,A,B$ are collinear in the order $P,A,B$ or $A,B,P$, then $\angle APB=0$.
Conversely, if $\angle APB=0$, then $P,A,B$ are collinear.
\end{angleaxiom}

\begin{angleaxiom}[Scaling invariance]\label{ax:A4}
If $A'\in \overrightarrow{PA}$ and $B'\in \overrightarrow{PB}$, then
\[
\angle APB=\angle A'PB'.
\]
\end{angleaxiom}

\begin{angleaxiom}[Continuous bisection]\label{ax:A5}
Fix a point $P$. Let $\mathcal R_P$ be the topological space of oriented rays emanating from $P$,
and let $S_P\subset\mathcal R_P$ be a singular subset.
Set $D_P:=\mathcal R_P\setminus S_P$.
For each connected component $\gamma$ of $D_P$, there exists a map
\[
\angle:\gamma\times \gamma\to\mathbb{R},
\]
called the angle map, satisfying:
\begin{enumerate}
\item[(i)] (Angle bisection)
For any $r,s\in \gamma$, there exists $t\in \gamma$ such that
\[
\angle(r,t)=\angle(t,s)=\tfrac12\,\angle(r,s).
\]
\item[(ii)] (Continuity)
For each fixed $r\in \gamma$, the map
$s\mapsto\angle(r,s)$ is continuous on $\gamma$.
\end{enumerate}
\end{angleaxiom}

\begin{convention}[Boundary policy (with covering groups)]\label{conv:boundary}
The axioms in this paper are defined on $D_P=\mathcal R_P\setminus S_P$.
The behavior on $S_P$ is not axiomatized.
When necessary, each geometry specifies a value space $V$ and a covering map
$p:\mathbb{R}\to V$ with deck transformation group
$\Lambda \le (\mathbb{R},+)$, assumed discrete,
and adopts one or more of the following conventions:
\begin{itemize}
  \item \textbf{Lift type:}
  Choose a continuous lift
  $\widetilde{\Angle}:\gamma\times \gamma\to\mathbb{R}$ on each connected component $\gamma$ of $D_P$,
  and define $\Angle=p\circ\widetilde{\Angle}$.
  The lift $\widetilde{\Angle}$ is $\Lambda$--periodic
  ($\widetilde{\Angle}\sim\widetilde{\Angle}+\lambda$, $\lambda\in\Lambda$).
  Taking $\Lambda=\{0\}$ includes nonperiodic real--valued angles.
  \item \textbf{Absorption type:}
  Singular directions are treated as being absorbed,
  e.g.\ $\Angle(r,s)\to 0$ as $s$ approaches $S_P$.
  The value space may be $V=\mathbb{R}$ or $\overline{\mathbb{R}}$.
  \item \textbf{Divergence type:}
  One sets $\Angle(r,s)\to\pm\infty$ as $s$ approaches $S_P$,
  without continuous extension.
  The value space is $V=\overline{\mathbb{R}}$.
\end{itemize}
Different policies may be adopted for different connected components
or types of singularities.
\end{convention}

\bibliographystyle{plain}
\bibliography{diff_angle_ref}


\section*{Acknowledgements}

The author would like to express his heartfelt gratitude to his wife, Junko,
whose constant support sustained the writing of this paper.

I am deeply grateful to my friend Kanato Takeshita, whom I came to know during my graduate studies at Tohoku University, for his candid opinions and advice, which helped clarify the direction of this research.

I would also like to express my sincere gratitude to Satoshi Ishiwata (Yamagata University), as well as to my friends Mio Yamakawa and Ryohei Sakai, for kindly providing opportunities for careful discussions after the submission of Base1(v1), and for offering many valuable comments and suggestions.

As a high school teacher, the author does not always have sufficient time
to devote to research.
In this context, AI tools that supported comparative studies of existing
geometric frameworks, the construction of a \TeX\ environment,
and editorial organization played an indispensable role,
and the author gratefully acknowledges this assistance.

The author also wishes to thank Professor Athanase Papadopoulos and
Professor Charalampos Charitos for their warm advice and encouragement
during the arXiv submission process,
as well as Professor Jean--Marc Schlenker for endorsing the submission
and thereby making the public dissemination of this work possible.
Their generous support gave the author the confidence to present his
independent research publicly and provided strong motivation to continue
this work despite limited available time.

Furthermore, the author's considerations on parabolic trigonometric functions
and Brocard points were inspired by the work of Professor Gunter Weiss
and Professor Boris Odehnal.
During a careful re-examination of the definitions after the initial arXiv submission,
the author experienced serious doubts about their adequacy,
which led to a fundamental reconsideration of the structure of this paper.
Through insights gained from their work,
the author came to recognize a remarkably elegant connection
between Brocard points, parabolic trigonometric functions,
and circular trigonometric functions.
This experience constituted a decisive turning point in completing the present work,
and the author expresses his sincere gratitude here.

Finally, the author wishes to express his profound respect and gratitude
to Felix Klein, David Hilbert, and many other great mathematicians whose
ideas have influenced this research, as well as to all researchers who
offered advice during the endorsement process.
Although the author has not yet attained a complete understanding of
geometry itself, the process of this research has allowed him to encounter
a wide variety of geometric frameworks and to appreciate the importance
of viewing mathematics from multiple perspectives.

\end{document}